\documentclass[10pt,reqno]{amsart}

\usepackage{amsfonts}
\usepackage{amssymb}
\usepackage{amsthm}
\usepackage{amsmath}
\usepackage{graphicx}

\usepackage[all]{xy}

\theoremstyle{definition}
 \newtheorem{theorem}{\bf Theorem}[section]
 
 \newtheorem{lem}[theorem]{Lemma}
 \newtheorem{cor}[theorem]{Corollary}
 
\theoremstyle{definition}
 \newtheorem{example}[theorem]{Example}
 \newtheorem{remark}[theorem]{Remark}
 \newtheorem{definition}[theorem]{Definition}
 \newtheorem{prop}[theorem]{Proposition}
\numberwithin{equation}{section}

\newcommand{\bb}[1]{\mathbb{#1}}
\newcommand{\M} {\mathcal{M}}

\author[A. Marinkovi\'c \and Milena Pabiniak]{Aleksandra Marinkovi\'c \and Milena Pabiniak}
\address{Centro de An\'alise Matem\'atica, Geometria e Sistemas Din\^amicos, Instituto Superior T\'ecnico, Av. Rovisco Pais, 1049-001 Lisboa, Portugal}
\email{aleksperisic(at)yahoo.com \and milenapabiniak(at)gmail.com}

\title[Centered reduction]{Every symplectic toric orbifold is a centered reduction of a Cartesian product of weighted projective spaces}
\begin{document}

\maketitle

\begin{abstract}
We prove that every symplectic toric orbifold is a centered reduction of a Cartesian product of weighted projective spaces. A theorem of Abreu and Macarini shows that if the level set of the reduction passes through a non-displaceable set then the image of this set in the reduced space is also non-displaceable. Using this result we re-prove that every symplectic toric orbifold contains a non-displaceable fiber and we identify this fiber.
\end{abstract}

%%%%%%%%%%%%%%%%%%%%%%%%%%%%%%%%%%%%%%%%
\section{Introduction}\label{section introduction}
A {\bf symplectic toric manifold (orbifold)} is a symplectic manifold (orbifold) $(\M,\omega)$ 
equipped with an effective Hamiltonian action of a torus $T$, with $\dim \M = 2 \dim T$. 
Each symplectic toric orbifold has an associated moment map, $\mu:\M \rightarrow  \mathfrak{t}^*$, 
to the dual of the Lie algebra of $T$, which is unique up to translation.

If $\M$ is compact then the image of $\mu$ is a compact convex polytope in $ \mathfrak{t}^*$, 
the convex hull of the images of the fixed points of the action (Atiyah \cite{Atiyah:1982}, Guillemin, Sternberg \cite{guillstern:1982}). 
Moreover, the polytope is Delzant, i.e. it is simple, rational and smooth (see Section \ref{section shrinking} for precise definitions). Delzant in \cite{Delzant:1988} proved that compact symplectic toric manifolds, up to equivariant symplectomorphism 
 (definition in Section \ref{section shrinking}), are  classified by their moment map images, up to translation in $\mathfrak{t}^*$.
 In the proof he uses the data encoded in the moment polytope to show that every such manifold can be obtained from $\bb{C}^d$, with the standard symplectic structure, by a symplectic reduction (for some $d \in \bb{N}$).
This result was generalized to compact symplectic toric orbifolds by Lerman and Tolman \cite{LermanTolman:1997}. Compact symplectic toric orbifolds (up to equivariant symplectomorphism)
  can be classified by compact, convex, rational and simple polytopes (not necessarily smooth) with positive integers attached to each facet (up to translation, keeping labels fixed). 
We will call a convex rational polytope with positive integer labels attached to each facet a {\bf labeled polytope}. Compact symplectic toric orbifolds are also symplectic reductions of $\bb{C}^d$.

Thanks to the above correspondence between labeled polytopes and symplectic toric orbifolds we can study the combinatorial properties of polytopes in order to obtain some geometric information about the associated toric orbifolds. Note that $ \mathfrak{t} ^*$ is isomorphic (though not canonically) to $\mathbb{R}^{n}$, where $n$ denotes the dimension of $T$.
In Section \ref{section shrinking} we specify an identification of $ \mathfrak{t} ^*$ with $\mathbb{R}^n$ that we use throughout the paper.
This allows us to think about $\mu (\M)$ as a subset of $\mathbb{R}^n$.
Every convex rational polytope is an intersection of some number of half spaces, and it can be uniquely written as
$\Delta=\bigcap_{i=1}^d\{x\in\mathbb{R}^n |\langle x,w_i\rangle\leq\,l_i\},$
where $d$ is the number of facets, $\,l_i$ are real numbers, and the vectors $w_i\in\mathbb{Z}^n$ are primitive outward normals to the facets of the polytope. We denote by $\langle \cdot,\cdot\rangle:\mathbb{R}^n\times\mathbb{R}^n\rightarrow\mathbb{R}$ the standard inner product. Moreover, if $\Delta$ is a labeled polytope, then the labels $a_i\in \mathbb{N}$ can be incorporated in the description of the polytope by presenting the above half spaces as $\{x\in\mathbb{R}^n |\langle x,a_i w_i\rangle\leq\,a_i l_i\}$. Hence, every convex rational labeled polytope can be uniquely presented as $$\Delta=\bigcap_{i=1}^d\{x\in\mathbb{R}^n |\langle x,v_i\rangle\leq\lambda_i\},$$
where  
$v_i=a_i w_i$ and $w_i$ is the primitive outward normal, $\lambda_i=a_i \,l_i$ and $a_i \in \bb{N}$ is the label on the corresponding facet.
Therefore the polytope has a trivial labeling (all labels equal to $1$) if and only if all the $v_i$'s are primitive. We say that a labeled polytope is {\bf monotone} if $\lambda_i=\lambda$ for every $i=1,\ldots, d.$ Note that changing the labels may change monotonicity (see Figure \ref{figurezeroexample}).

In Section \ref{section shrinking}, motivated by the work of Reid \cite{Reid:1983} and the toric minimal model program of Gonzales and Woodward \cite{Gonzales.Woodward:2012}, we define a procedure of shrinking a labeled polytope $\Delta \subset \mathbb{R}^n$ to a point. The idea is to move each facet of $\Delta$ inward by reducing each $\lambda_i$ by the same amount. As a result we either get a point or a lower dimensional polytope. In the second case we carry on shrinking this lower dimensional polytope, and continue till we get to a point. We denote by $M+1$ the number of ``dimension drops'' that occur while shrinking. 
The Lagrangian torus fiber which is the preimage (under the moment map) of the final point of the shrinking procedure is called the {\bf central fiber}. Note that this definition does not depend on the choice of identification $\mathfrak{t}^*\cong \bb{R}^n$. 
If this final point is the origin, the polytope is called {\bf centered at the origin}. Note that every monotone polytope is centered at the origin and that
any polytope can be made centered at the origin by an appropriate translation (adding a constant to the chosen moment map). Thus we will always assume that the moment map image $\mu(\M)=\Delta$ is a polytope centered at the origin.
With this assumption the central fiber is $\mu^{-1}(0)$, the preimage of the origin $0$. 

We are ready to state the first theorem.
\begin{theorem}\label{prva}
Every (compact, convex, rational) centered at the origin, labeled polytope is an intersection of (compact, convex, rational) monotone labeled polytopes in one of the following ways:
\begin{equation}\label{broj2}
\Delta=\bigcap_{k=0}^N \widetilde{\Delta}^n_k
\end{equation}
if $M=0$ or
\begin{equation}\label{broj22}
\Delta=\bigcap_{k=0}^N \widetilde{\Delta}^n_k\hskip1mm\cap\bigcap_{j=1}^M(\widetilde{\Delta}^{k_1+\cdots+ k_j}\times\mathbb{R}^{n-(k_1+\cdots+ k_j)})
\end{equation}
if $M>0$, where $\dim\widetilde{\Delta}^n_k=\dim\Delta=n$, for every $k\in\{0,\ldots, N\},$ and $\dim\widetilde{\Delta}^{k_1+\cdots +k_j}=k_1+\cdots+ k_j$ for every $j\in\{1,\ldots, M\}$.
\end{theorem}
We obtain the information needed to construct the polytopes $\widetilde{\Delta}^n_k,\widetilde{\Delta}^l$  in Theorem \ref{prva}  by analyzing how the face structure of $\Delta$ changes while shrinking.
\begin{remark} $\Delta$ being smooth or simple does not imply that the polytopes on the right hand side of \eqref{broj2} and \eqref{broj22} are either smooth or simple. See Examples \ref{inter.newfirstexample} and \ref{inter.verynewfirst}.
\end{remark}
In our next result we use purely combinatorial Theorem \ref{prva} to present each symplectic toric orbifold as a ``centered" reduction of a product of weighted projective spaces, as we now explain.
\vskip1mm
A weighted projective space $\mathbb{CP}(m_1,\ldots,m_{d})$, with $(m_1,\ldots,m_d) \in \mathbb{Z}_+^d$ coprime, is a toric symplectic orbifold that is a symplectic reduction of $(\mathbb{C}^d,\frac{i}{2}\sum_{k=1}^ddz_k\wedge d\bar{z}_k)$ with respect to a circle action:
 $t*(z_1,\ldots,z_d)\rightarrow(t^{m_1}z_1,\ldots,t^{m_d}z_d).$ There is a toric $T^{d-1}$ action on $\mathbb{CP}(m_1,\ldots,m_{d})$, namely the residual action coming from the standard $T^d-$action on $\mathbb{C}^d.$ When at least one of the weights $m_j$ is equal to 1 (say $m_d=1$), then the corresponding polytope (centered at the origin) is
%$P =\mu (\mathbb{CP}(m_1,\ldots,m_{d}))$
 $$P=\bigcap_{i=1}^{d-1}\{x\in\mathbb{R}^{d-1}|\hskip1mm\langle x,-e_i\rangle\leq\lambda\}\cap\{x\in\mathbb{R}^{d-1}|\hskip1mm\langle x,(m_1,\ldots,m_{d-1})\rangle\leq \lambda\}$$
 where $e_i,\,i \in\{1,\ldots,d-1\}$ are coordinate vectors and $\lambda$ is a positive real number responsible for rescaling the symplectic form. If all the weights are trivial, i.e. $m_1=\ldots=m_d=1$ we obtain a toric symplectic manifold $\bb{CP}^{d-1}$. The central fiber of $\mathbb{CP}(m_1,\ldots,m_{d})$ is 
 $$T_{0}=\{[z_1,\ldots,z_{d}]\in\mathbb{CP}(m_{1},\ldots,m_{d})|\hskip1mm |z_1|^2=\cdots=|z_{d}|^2\}.$$ The central fiber in a Cartesian product of weighted projective spaces is the product of the central torus fibers of each weighted projective space. 

 \begin{definition}
 A reduction of a Cartesian product of weighted projective spaces is called a {\bf centered reduction} if it goes through the central torus fiber of the product.
\end{definition}
Any symplectic toric orbifold can be presented as a symplectic reduction of one weighted projective space, but this reduction is not necessarily centered (see Proposition \ref{mon.symp.red}).
It is worth looking for a centered reduction because often important symplectic information, for example, non-displaceability (see \cite[Corollary 3.4(i)]{AbreuMacarini:2013}), or  quasimorphisms (see \cite[Theorem 1.1]{Borman:2013}), is preserved under centered reduction.
 Our second result is the following.
\begin{theorem}\label{druga}
Every compact symplectic toric orbifold is a centered reduction of a Cartesian product of weighted projective spaces.
\end{theorem}
The idea of the proof is to use the presentation of the corresponding labeled polytope described in Theorem \ref{prva}. The number of monotone polytopes in the presentation will be the number of weighted projective spaces in the Cartesian product.

We use a centered reduction as a tool to prove the existence of a non-displaceable torus fiber in a reduced space.
A Lagrangian submanifold $L$ in a symplectic manifold $\M$ is {\bf non-displaceable} if for every Hamiltonian diffeomorphism $\varphi \colon \M \rightarrow \M$ we have $\varphi(L)\cap L\neq\emptyset.$
In \cite[Corollary 3.4(i)]{AbreuMacarini:2013}, Abreu and Macarini prove that if $\M$ is a symplectic reduction $\M:= \Phi^{-1}(a)/ T^k$ and  $\Phi^{-1}(a)$ contains a non-displaceable $T^k$-invariant set then the image of this non-displaceable set is non-displaceable in the reduced space $\M$. The central fiber of the Cartesian product of weighted projective spaces is non-displaceable (\cite{ChoPoddar:2012}, \cite{Gonzales.Woodward:2012}; see also Section \ref{non-displ}). Therefore we obtain the following result.
\begin{theorem}\label{posled1} 
For every compact symplectic toric orbifold its central fiber is non-displaceable.
\end{theorem}
The above fact is already known (see the works of Fukaya-Oh-Ohta-Ono \cite{fooo:2010}, \cite{fooo:2008} using vanishing of Lagrangian Floer cohomology, and the works of E. Gonzales and C. Woodward \cite{Woodward:2011}, \cite{WilsonWoodward:2012}, \cite{Gonzales.Woodward:2012}, using quasimap Floer cohomology). Our proof of Theorem \ref{posled1} does not involve any calculation in Floer cohomology (though it uses the non-trivial fact that the central torus fiber of a weighted projective space is non-displaceable).

We finish the paper with another application of Theorem \ref{prva}.
Combinatorial analysis of the moment polytopes of symplectic toric manifolds, done in the proof of Theorem \ref{prva}, allows us to give some lower and upper bounds on the Gromov width of these manifolds (see Section \ref{section gromovwidth} for precise definitions and results).

{\bf Organization.} In Section 2, we describe the shrinking procedure. We use this procedure to prove Theorem \ref{prva} in Section 3. Section 4 contains a description of properties of centered reduction and the proof of Theorem \ref{druga}. As a corollary of this Theorem we prove non-displaceability of the central fiber in every compact symplectic toric orbifold. Application of Theorem \ref{prva} to the questions about the Gromov width is described in Section \ref{section gromovwidth}. 

{\bf Acknowledgements.} We would like to thank Professor Miguel Abreu who suggested this problem to us, and to Strom Borman, Felix Schlenk, Kristin Shaw and Benjamin Assarf for helpful discussions.
Moreover we thank the anonymous referees for their comments which have improved an earlier version of this paper.

The authors were supported by the Funda\c{c}\~ao para a Ci\^encia e a Tecnologia (FCT, Portugal):
fellowships SFRH/BD/77639/2011 (Marinkovi\'c) and 
\\SFRH/BPD/87791/2012 (Pabiniak);
projects PTDC/MAT/117762/2010 (Marinkovi\'c and Pabiniak) and 
EXCL/MAT-GEO/0222/2012 (Pabiniak).

%%%%%%%%%%%%%%%%%%%%%%%%%%%%%%%%%%%%%%%%%%%%%%%%%%%%%%%%%%%%%%%%%%%%%%%%%%%
\section{Shrinking procedure}\label{section shrinking}

In this section we study labeled polytopes by analyzing their behavior during a procedure of ``shrinking", defined below.
A polytope is called:
\\ $\bullet$ { \bf simple} if there are $n$ edges meeting at each vertex $V$,
 \\ $\bullet$ { \bf rational} if all the edges meeting at a vertex $V$ are of the form $V+t\nu_i,\hskip1mm t\geq0$ where $\nu_i\in  \mathfrak{t}_{\mathbb{Z}}^*$ are primitive integral vectors
 \\ $\bullet$ {\bf smooth} if for each vertex, the corresponding $\nu_i,\hskip1mm i=1,\ldots,n$ form a $\mathbb{Z}-$ basis of $ \mathfrak{t}_{\mathbb{Z}}^*$. 
 \\Convex polytopes satisfying the above three conditions are called {\bf Delzant polytopes}. These polytopes classify symplectic toric manifolds in the following sense.
 Moment map image of a symplectic manifolds equipped with a toric $T$ action gives a Delzant polytope in  $ \mathfrak{t}^*$. Two such manifolds, $\M_1$ and $\M_2$, are called {\bf equivariantly symplectomorphic} if there exists a symplectomorphism $\phi \colon \M_1 \rightarrow \M_2$ such that $\phi( t \cdot x)=t \cdot (\phi(x))$ for all $t \in T$ and all $x \in \M_1$. The moment map gives a bijection between symplectic manifolds with a toric $T$ action, up to equivariant symplectomorphism, and the set of Delzant polytopes in  $ \mathfrak{t}^*$, up to translation (\cite{Delzant:1988}).
To simplify the notation we will now fix an identification of $\mathfrak{t}^*$ with $\bb{R}^n$ and work in $\bb{R}^n$, $n=\dim T$.
\\{\bf Conventions:} Fix any splitting, $T\cong (S^1)^n$, of the torus acting and let $S^1=\bb{R}/2 \pi \bb{Z}$, i.e. the exponential map $\exp \colon Lie(S^1)\cong \bb{R} \rightarrow S^1$
is given by $t \mapsto e^{i t}$.
Choosing a different splitting of the torus $T$ into a product of circles corresponds to applying a $\pm SL(n,\bb{Z})$ transformation to $\bb{R}^n$. Changing the convention to $S^1 \cong \bb{R}/  \bb{Z}$ would result in rescaling by a factor of ${2 \pi}$. 
With the identification $\mathfrak{t}^* \cong \bb{R}^n$ fixed we can now assign to each compact symplectic toric orbifold $\M$
a labeled, rational and simple polytope $\Delta \subset \bb{R}^n$ being the moment map image. If $\M$ is smooth then $\Delta$ is also smooth. As a moment map is unique only up to translation, this polytope is also unique only up to translation.
We later specify a choice of moment map, making the associated polytope unique. See Remark \ref{choose mm}.

Let
\begin{equation}\label{polytope}
  \Delta=\bigcap_{i=1}^d\{x\in\mathbb{R}^n |\langle x,v_i\rangle\leq\lambda_i\}
\end{equation}
be a labeled $n$-dimensional polytope with $d$ facets.
That is, each vector $v_i$, $i=1,\ldots, d$ does not need to be primitive but may be an $a_i$ multiple of a primitive vector $w_i$, for some 
positive integer $a_i$.

 Following the ideas in \cite{Gonzales.Woodward:2012} and \cite{Reid:1983} we define a purely combinatorial procedure of shrinking a convex polytope. The idea is to shrink the polytope continuously by moving each facet inward with the same speed, that is by reducing each $\lambda_i$ by $t$ at the time $t$ till we get to one point. It might happen after some time of shrinking that some half-spaces intersect only in a lower dimensional $\bb{R}^k$ and if we move them a tiny bit more we obtain an empty set of intersections. If this happens we stop moving these facets and carry on moving the remaining ones. Below is the precise definition. For any $t>0$ let
$$H_i^t:=\{x \in \bb{R}^n\,|\, \langle x,v_i\rangle=\lambda_i-t\}, \,\,\,\,h_i^t:=\{x \in \bb{R}^n\,|\, \langle x,v_i\rangle \leq\lambda_i-t\}.$$
For $t>0$ small enough the set
$$\Delta^t:= \{x \in \bb{R}^n\,|\, \langle x,v_i\rangle \leq\lambda_i-t;\,i=1,\ldots,d\}$$
is an $n$-dimensional polytope. Let $t_{1}$ be the smallest $t$ for which $\dim \Delta^{t}=n-k_1 < n$. If $k_1=n$ then $\Delta^{t_{1}}$ is a point and we stop the shrinking prodecure. Suppose $k_1< n$. That means that there exists an orthogonal  splitting $\bb{R}^n=\bb{R}^{k_1}\times \bb{R}^{n-k_1}$ and a point $p\in\mathbb{R}^n$ such that $ \Delta^{t_{1}} \subset p+\{0\} \times  \bb{R}^{n-k_1}$. We can take $p$ such that $p+\{0\}\times B^{n-k_1}(\delta) \subset \Delta^{t_{1}}$ for some small $\delta >0$.
  Define
$$\mathcal{D}_1:=\{i \in \{1,\ldots,d\}\,|\,  \Delta^{t_{1}} \subset H_i^{t_{1}}\}.$$
 \begin{lem} $\mathcal{D}_1$ is a non-empty set and for every $i \in \mathcal{D}_1$,\hskip1mm $v_i \in \mathbb{R}^{k_1}\times\{0\}$.
 \end{lem}
 \begin{proof} If $\Delta^{t_{1}}$ is a point, i.e. if $k_1=n$ the claim is obvious. Assume that $k_1<n$. Note that then there exists $i$ such that $\langle p,v_i \rangle =\lambda_i-t_{1}$, i.e. $p \in H_i^{t_{1}}$. Indeed, suppose not and take any
$0 \neq v \in \bb{R}^{k_1}\times\{0\}$ and any $i \in\{1,\ldots,d\}$. As $p  \notin H_i^{t_{1}}$ but  $p  \in h_i^{t_{1}}$ there exists some $\varepsilon_i>0$ such that $p+\varepsilon_i v \in h_i^{t_{1}}$. Let $\varepsilon:=\min\{\varepsilon_1,\ldots, \varepsilon_d\}$. Then $$p+ \varepsilon v \in \Delta^{t_{1}}\subset p+ \{0\} \times \bb{R}^{n-k_1}$$ though $p+ \varepsilon v  \notin p+ \{0\} \times \bb{R}^{n-k_1}$. Contradiction.
\\Take any $i$ such that $p \in H_i^{t_1}$ and any $w \in \{0\}\times B^{n-k_1}(\delta)$.  Then $p+w,p-w \in  \Delta^{t_1} \subset  h_i^{t_1}$, so
$\langle p+w,v_i\rangle $, $ \langle p-w,v_i\rangle \leq\lambda_i-t_{1}$. This implies $\langle w,v_i\rangle $, $ \langle -w,v_i\rangle \leq\lambda_i-t_{1}-\langle p,v_i\rangle=0$ and therefore $\langle w,v_i\rangle =0$ for any $w \in \{0\}\times B^{n-k_1}(\delta)$. We conclude that $v_i \in \mathbb{R}^{k_1}\times\{0\}$ and thus for any $q \in \Delta^{t_1} \subset  \{0\} \times  \bb{R}^{n-k_1}$ we have $\langle q,v_i\rangle = \langle p,v_i\rangle =\lambda_i-t_{1}$. This proves that  $\Delta^{t_1} \subset H_i^{t_1}$.
\end{proof}
From the moment $t_{1}$ we shrink only the facets $f_i$ with $i \in \{1,\ldots,d\} \setminus \mathcal{D}_1$
 and proceed similarly. Let $t_{2}< \cdots <t_{{M+1}}$ be the times when dimension drops while shrinking and let $k_j$ denote the value by which dimension drops at a time $t_{j}$. At the time $t_{M+1}$ our polytope has shrunk to a single point. We define non-empty sets
 $$\mathcal{D}_s:=\{i \in \{1,\ldots,d\} \setminus (\mathcal{D}_1 \cup \ldots \cup \mathcal{D}_{s-1})\,|\, \Delta^{t_{s}} \subset H_i^{t_{s}}\}, s=1,\ldots M+1.$$
Therefore for $t \in (t_{j},t_{j+1}]$ we have
\begin{align*} \Delta^t:= \{x \in \bb{R}^n\,|\, & \langle x,v_i\rangle \leq\lambda_i-t,\,\,\,i \in \{1,\ldots,d\}\setminus (\mathcal{D}_1 \cup \ldots \cup \mathcal{D}_j),\\
&
\langle x,v_i\rangle =\lambda_i-t_{s}\,\,\,i \in \mathcal{D}_s,\,s=1,\ldots,j \}.\end{align*}

Hence, $\dim \Delta^{t}=n-(k_1+\ldots +k_j)$  for some $k_1,\ldots,k_M \geq 1$, and $k_{M+1}=n-(k_1+\ldots +k_M)\geq 1$.
\begin{definition} A labeled polytope is called {\bf centered at the origin} if a single point obtained at the end of the shrinking procedure is the origin.
\end{definition}
By changing the labels of a centered polytope we may obtain a polytope that is not centered (see Figure \ref{figurezeroexample}).
\begin{figure}
\centering
\includegraphics[width=8.5cm]{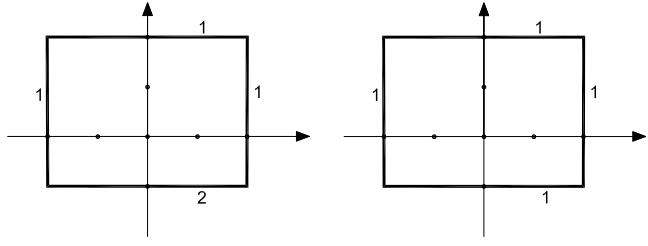}
\caption{Two labeled polytopes that differ only by labels: monotone and centered (on the left), neither monotone nor centered (on the right). }
\label{figurezeroexample}
\end{figure}
Note that, if $\Delta$ is centered polytope, then for every $j\in\{1,\ldots,M+1\}$ we have
$$\mathcal{D}_j=\{i\in\{1,\ldots,d\}|\, \lambda_i=t_{j}\}.$$
\begin{remark}\label{choose mm}
Every labeled polytope can be translated by a vector to a centered position. This does not change the corresponding orbifold, as the moment map is defined up to translation.
Therefore from now on we always choose the moment map $\mu$ for which $\mu(\M)=\Delta$ is centered at the origin. 
With this assumption the central fiber is the preimage (under the moment map) of $0$. 
\end{remark}

There exists an orthogonal splitting $\bb{R}^n=\bb{R}^{k_1}\times \ldots \times \bb{R}^{k_{M+1}}$ such that for any $j=1,\ldots, M+1$ the polytope $\Delta^{t_{j}} \subset \{0\}\times \bb{R}^{k_{j+1}}\times \ldots \times\bb{R}^{k_{M+1}}$.
Let 
     \begin{align}\label{projections}
      \pi_j &\colon \bb{R}^n \rightarrow  \bb{R}^{k_{j+1}}\times \ldots \times\bb{R}^{k_{M+1}}=\bb{R}^{n-(k_1+\ldots +k_j)},\nonumber \\ 
\pi_j^{\perp} &\colon \bb{R}^n \rightarrow  \bb{R}^{k_{1}}\times \ldots \times\bb{R}^{k_{j}}
     \end{align}
be orthogonal projections. Then, $\pi_j(v_i)=0,$ for every $i\in\mathcal{D}_1\cup\cdots\cup\mathcal{D}_j$. We view $\Delta^{t_{j}}$ as a full dimensional polytope in $\bb{R}^{n-(k_1+\ldots +k_j)}$ with normals $\pi_j(v_i)$ for some set of $i$'s in $ \{1,\ldots,d\} \setminus (\mathcal{D}_1 \cup \ldots \cup \mathcal{D}_{j})\}$. Note that even if $v_i$ was primitive $\pi_j(v_i)$ does not need to be. This means that during the shrinking procedure the labels may change at the time when the polytope's dimension drops and this can happen only then. In particular, if the original polytope $\Delta$ has a trivial labeling, then $\Delta^t$ may not have a trivial labeling.

\subsection{Examples}
Below we give some explicit examples of the shrinking procedure. During this procedure two types of events can occur and change the combinatorial type of the polytope. The first one is the dimension drop already described above. The other event is what could be called ``disappearing'' of facets. It occurs when the inequality $\langle x,v_i  \rangle \leq\lambda_i-t$ for some $i$ becomes superfluous, that is, it is already implied by other inequalities defining $\Delta^t$. Note that the combinatorial type of a monotone polytope does not change during the shrinking porcedure, until we shrink to the point.

\begin{example}\label{newfirstexample} Here is an example of a simple polytope with a trivial labeling, having only one dimension drop (Figure \ref{figure newfirstexample}).
\begin{align*} \Delta = \{x\in\mathbb{R}^2|&\langle x,(0,\pm1)\rangle\leq3,\langle x,(-1,0)\rangle\leq3,\langle x,(1,\pm1)\rangle\leq3,
\langle x,(2,1)\rangle\leq4\}.\end{align*}
 At the time $t=2$ inequality $\langle x,(2,1)\rangle\leq4-t$ becomes superfluous and the face structure of polytope $\Delta^t$ is changing. We obtain a monotone polytope $\Delta^2$. The original polytope shrinks to zero in the time $t_1=3.$
 \begin{figure}\centering 
\includegraphics[width=5cm]{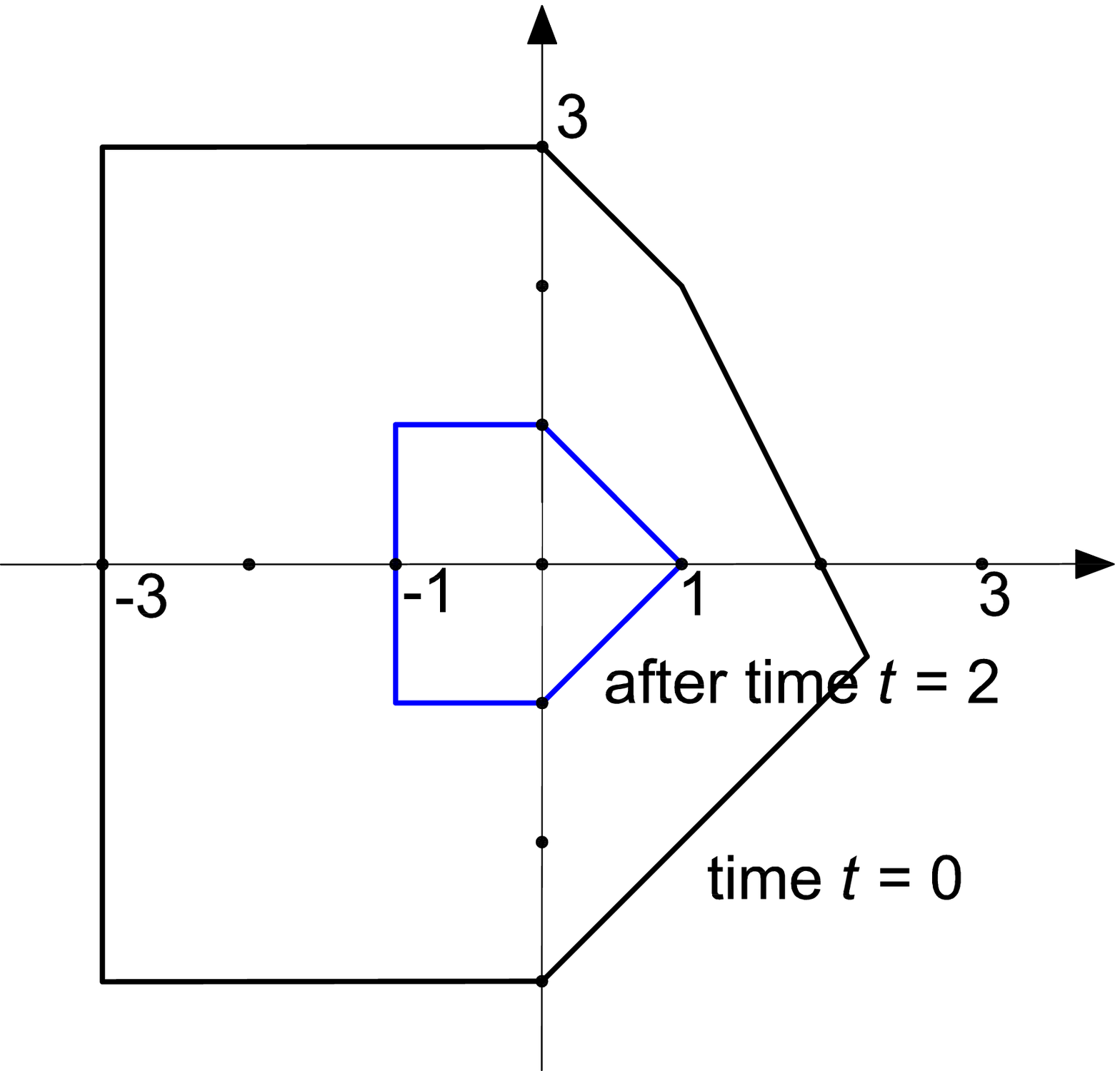}
\caption{Example \ref{newfirstexample}}
\label{figure newfirstexample}
\end{figure}
\end{example}

\begin{example}\label{verynewfirstexample} Here is an example of a simple polytope with a trivial labeling, having two dimension drops (Figure \ref{figure verynewfirstexample}).
\begin{align*} \Delta = \{x\in\mathbb{R}^3|&\langle x,(0,0,\pm1)\rangle\leq1,\langle x,(\pm1,0,1)\rangle\leq2,\langle x,(0,\pm1,1)\rangle\leq2\}.\end{align*}
At the time $t_1=1$ the polytope's dimension drops by 1 and we obtain a monotone polytope. The original polytope shrinks to zero in the time $t_2=2.$
\begin{figure}\centering 
\includegraphics[width=10cm]{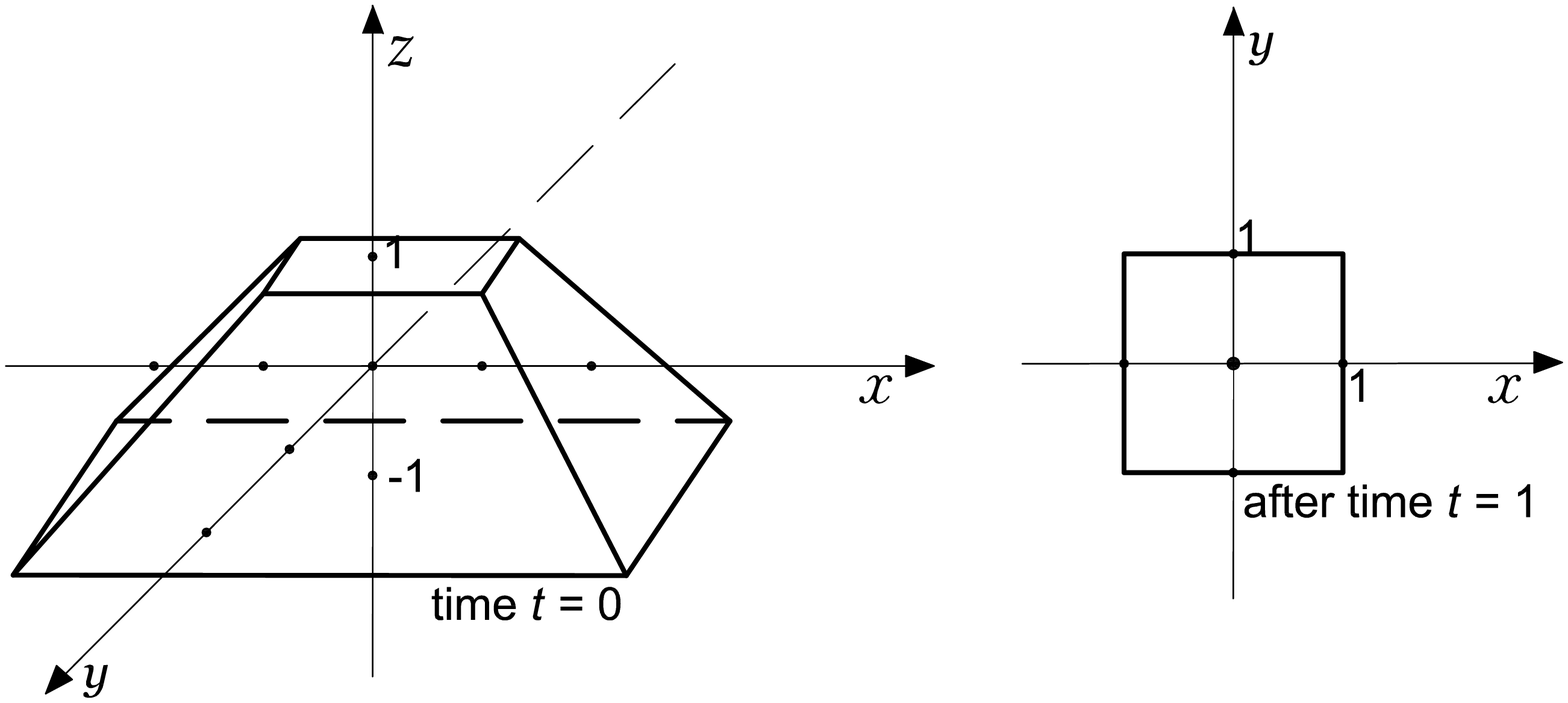}
\caption{Example \ref{verynewfirstexample}}
\label{figure verynewfirstexample}
\end{figure}
\end{example}
\begin{example}\label{firstexample} This example shows that the order in which facets disappear does not imply any inequalities between the corresponding coefficients $\lambda$'s.
\begin{align*} \Delta =\{x\in\mathbb{R}^2|&\langle x,\pm(1,0)\rangle\leq4,\langle x,\pm(0,1)\rangle\leq4,\langle x,(1,1)\rangle\leq6,
\langle x,(-1,1)\rangle\leq7,\\
&\langle x,(-2,-3)\rangle\leq12\}.\end{align*}
    At the time $t=1$ inequality $\langle x,(-1,1)\rangle\leq7-t$ becomes superfluous. At the time $t=2$ inequalities $\langle x,(1,1)\rangle\leq6-t$ and $\langle x,(-2,-3)\rangle\leq12-t$  become superfluous and we obtain a monotone polytope. The original polytope shrinks to zero in time $t_1=4.$ See Figure \ref{figure firstexample}.
 \begin{figure}\centering 
\includegraphics[width=6cm]{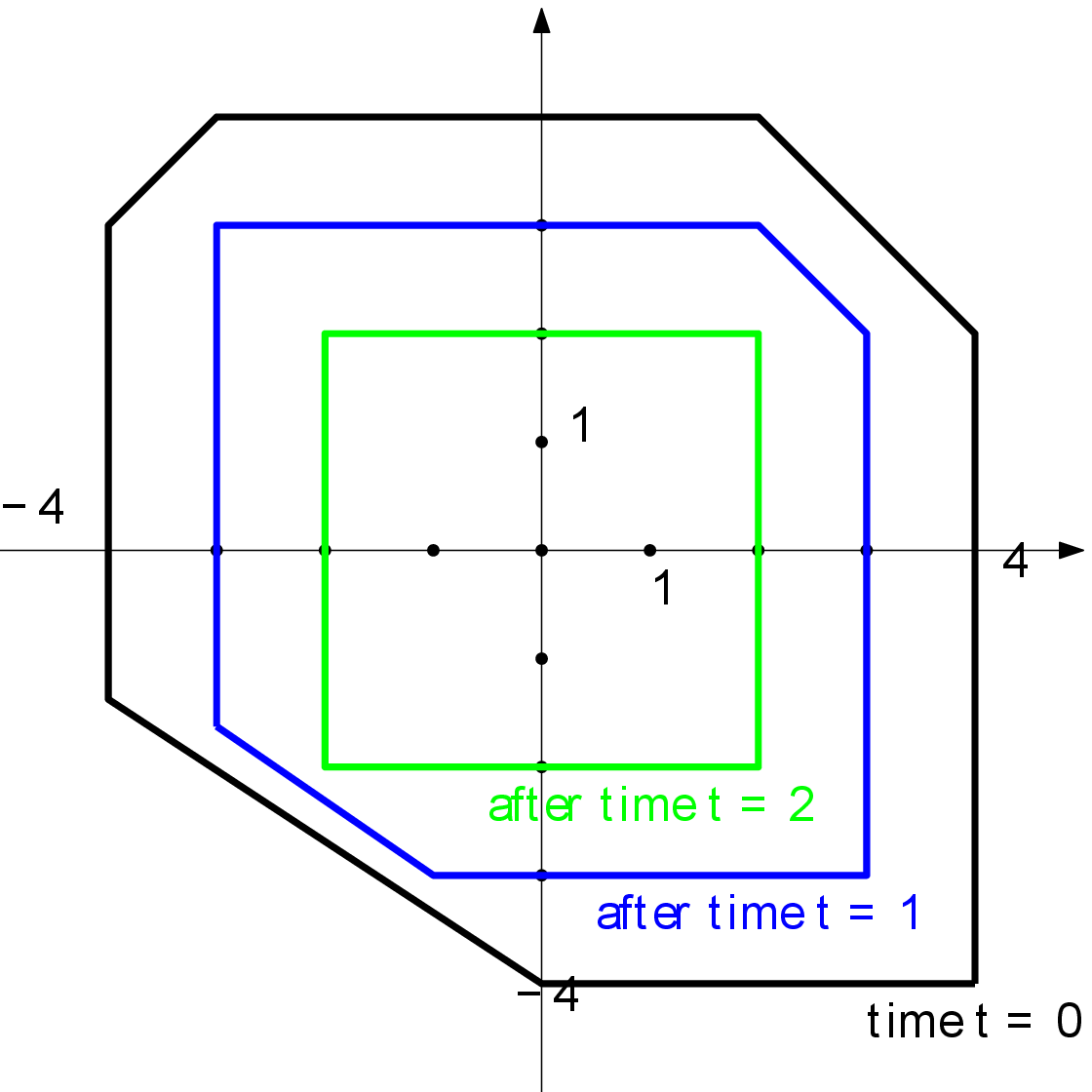}
\caption{Example \ref{firstexample}}
\label{figure firstexample}
\end{figure}
 \end{example}
 \begin{example}\label{secondexample} The following example shows that between two dimension drops some facets may ``disappear'' as well.
 \begin{align*} \Delta =\{x\in\mathbb{R}^2|&\langle x,\pm(1,0,0)\rangle\leq4,\langle x,\pm(0,1,0)\rangle\leq3,\langle x,(1,1,0)\rangle\leq5,
\langle x,\pm(0,0,1)\rangle\leq1\}\end{align*}
 At the time $t_1=1$ the dimension drops by 1. At the time $t=2$ one facet disappears, and at the time $t_2=3$ the dimension drops again and we obtain a monotone polytope. The original polytope shrinks to zero in the time $t_3=4.$ See Figure \ref{figure secondexample}.
 \begin{figure}\centering 
\includegraphics[width=11cm]{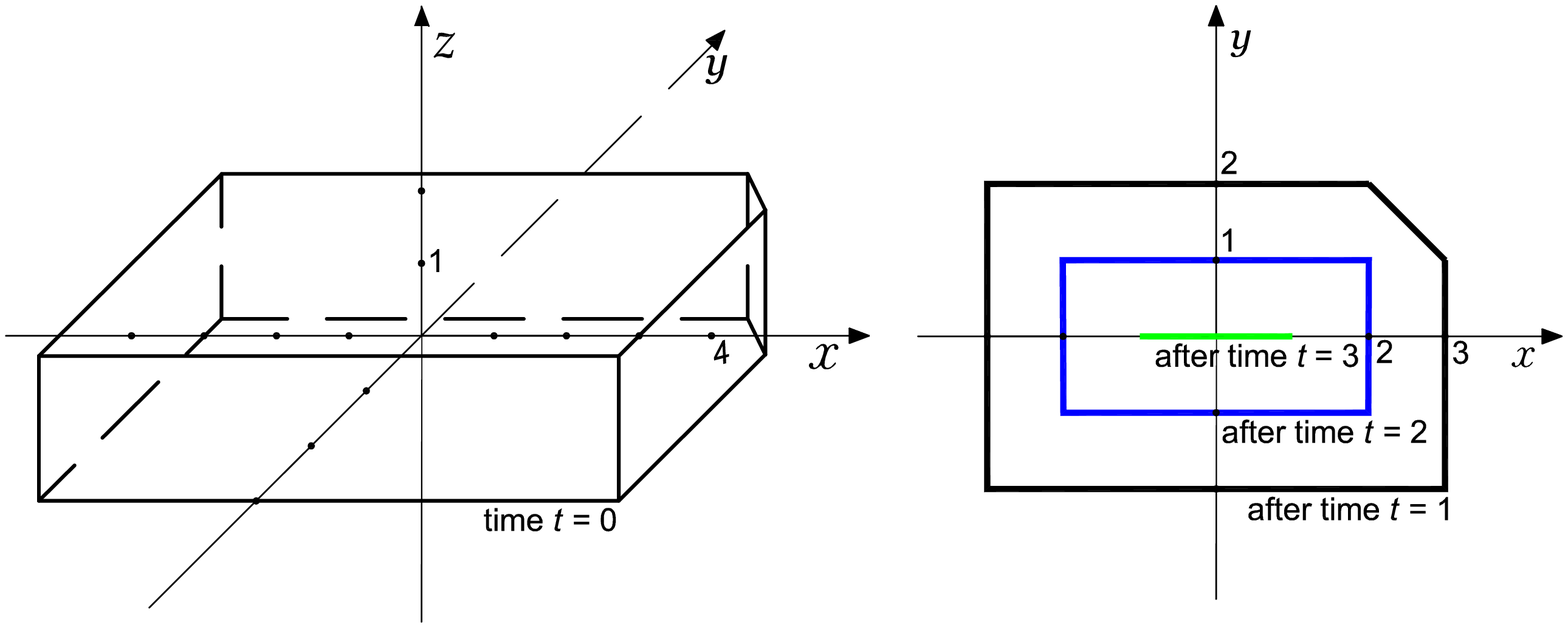}
\caption{Example \ref{secondexample}}
\label{figure secondexample}
\end{figure}
 \end{example}
The labeling of facets may change but only at the moments of a dimension drop. The reason is that at that moment the facet normals change from $w$ to $\pi(w)$, where $\pi$ is some orthogonal projection. Even if $w$ was primitive, $\pi(w)$ does not need to be, and therefore the label on the corresponding facet may change. Due to this change of labeling some facets ``move'' slower during the shrinking than expected. As a result, a facet that may seem to disappear at the moment of a dimension drop in fact carries a relevant information. This is why in the shrinking procedure we keep track of all the inequalities. Here is an example of this type of a situation.
\begin{example}\label{thirdexample} Consider a simple, rational, but not smooth polytope 
 \begin{align*} \Delta =\{x\in\mathbb{R}^2|&\langle x,\pm(0,1)\rangle\leq2,\langle x,(1,1)\rangle\leq6,\langle x,(2,-1)\rangle\leq8, 
\langle x,(-1,0)\rangle\leq6\}.\end{align*}
Note that at the time $t=1$ one facet seems to ``disappear''  as the inequality $\langle x,(1,1)\rangle\leq6-t$ is implied by other inequalities. However, the information encoded in that facet will become relevant after the time $t=4$.
 At the time $t=2$ the dimension drops by 1. The resulting one-dimensional polytope is given by the half spaces
$$h^2_3=\{x\in\mathbb{R}|\langle x,1\rangle\leq4\},
  h^2_4=\{x\in\mathbb{R}|\langle x,2\rangle\leq6\},
   h^2_5=\{x\in\mathbb{R}|\langle x,-1\rangle\leq4\}.$$
It is smooth and has a non-trivial labeling. At the time $t=4$ hyperplanes $H^4_3$ and $H^4_4$ meet and then, from the time $t=4$ the inequality $\langle x,2\rangle\leq 8-t$ is implied by the inequality $\langle x,1\rangle\leq6-t$ for every $t\geq4.$ That is, the facet $f_4$ disappears at the time $t=4$ and we obtain a monotone polytope defined by facets with a trivial labeling. The original polytope shrinks to zero in the time $t_2=6.$ See Figure \ref{figure3}.
 \begin{figure}\centering 
\includegraphics[width=11cm]{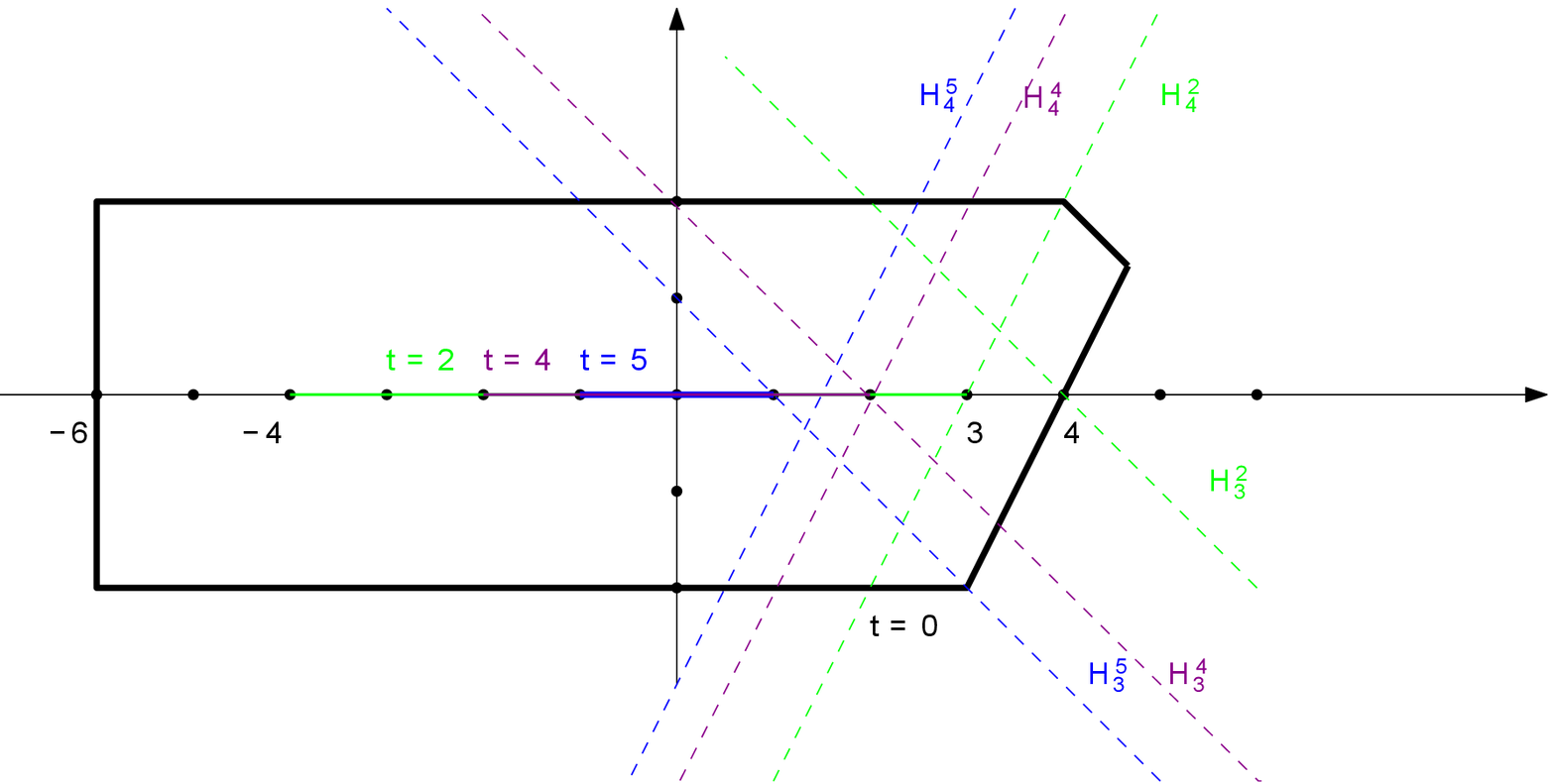}
\caption{Example \ref{thirdexample}.}\label{figure3}
\end{figure}
\end{example}
%%%%%%%%%%%%%%%%%%%%%%%%%%%%%%%%%%%%%%%%%%%%%%%%%%%%%%%%%
\section{Proof of Theorem \ref{prva}.} \label{section main proof}
In this section we prove Theorem \ref{prva}. We start by analyzing changes in the facet structure of $\Delta$ that happen during the shrinking procedure and use this information to define the polytopes appearing in Theorem \ref{prva}. Then we show that the sets defined in this way are indeed compact polytopes. At the last step we prove the equality of the sets claimed in Theorem \ref{prva}.

Let $M+1\geq 1$ be the number of dimension drops. That is, there are times $t_{1}<\cdots< t_{M}< t_{M+1}$ when the dimension drops by $k_1,\ldots,k_M, k_{M+1}\geq 1$ respectively, and $k_1+\cdots +k_M + k_{M+1}=n$. At the time $t_{M+1}$ we get a degenerate polytope $\Delta^{t_{M+1}}=\{0\}$.

\begin{lem}\label{small.lambda} For each  $j\in\{1,\ldots,M+1\}$ and $k\in\{1,\ldots,d\}\backslash(\mathcal{D}_1\cup\cdots\cup\mathcal{D}_{M+1})$ we have $t_{j}<\lambda_k$.
\end{lem}
\begin{proof}
Assume that for some $j\in\{1,\ldots,M+1\}$ and some $k\in\{1,\ldots,d\}\backslash(\mathcal{D}_1\cup\cdots\cup\mathcal{D}_{M+1})$ we have
 $\lambda_k\leq t_{j}.$ As $k \notin \mathcal{D}_{j}$ we have $\lambda_k \neq t_{j},$ so $\lambda_k < t_{j}.$
That means that during the shrinking procedure the hyperpalne $H_k$ meets the origin before a time $t_{j}$. Therefore $\{0\} \notin h_k^{t_{j}}$ and thus $\{0\} \notin \Delta^{t_{j}}$. This contradicts the assumption that $\Delta$ is centered.
 \end{proof}
Let $N\geq0$ be the number of different $\lambda_i$'s with $i \in \{1,\ldots,d\}\backslash(\mathcal{D}_1\cup\cdots\cup\mathcal{D}_{M+1})$, that is with $\lambda_i >t_{M+1}$. Assume that $\lambda_{j_1}<\cdots<\lambda_{j_N}$. We put the indices of these facets into groups
$$\mathrm{I}_k=\{i\in\{1,\ldots,d\}|\hskip1mm \lambda_i=\lambda_{j_k}\},\hskip2mm k=1,\ldots,N, \hskip2mm\mathrm{I}_0=\emptyset.$$

 Recall the notation from Section \ref{section shrinking}, $$\mathcal{D}_j=\{i\in\{1,\ldots,d\}|\,\lambda_i=t_{j}\},\hskip2mm j=1,...,M+1.$$
According to Lemma \ref{small.lambda} we have $t_{1}<\cdots<t_{M+1}<\lambda_{j_1}<\cdots<\lambda_{j_N}$
so all the sets $\mathrm{I}_k$ and $\mathcal{D}_j$ are disjoint.
Recall also the projections $\pi_j^\perp$ from equation \eqref{projections}.
\begin{definition}\label{definition polytopes}
We define
\begin{equation}\label{def.pol}
\begin{aligned}
 \widetilde{\Delta}^{n}_{0}&=\bigcap_{i\in\mathcal{D}_1\cup\cdots\cup\mathrm{D}_{M+1}}\{x\in\mathbb{R}^{n}|\hskip1mm\langle x,v_i\rangle\leq t_{M+1}\},\\
\widetilde{\Delta}^{n}_{k}&=\bigcap_{i\in\mathcal{D}_1\cup\cdots\cup\mathrm{D}_{M+1}\cup\mathrm{I}_k}\{x\in\mathbb{R}^{n}|\hskip1mm\langle x,v_i\rangle\leq\lambda_{j_k}\},\\
\widetilde{\Delta}^{k_1+\cdots+ k_j}&=\bigcap_{i\in\mathcal{D}_1\cup\cdots\cup\mathcal{D}_j}\{y\in\mathbb{R}^{(k_1+\cdots+k_j)}|\hskip1mm\langle y,\pi_j^\bot(v_i)\rangle\leq t_{j}\}\end{aligned}
\end{equation}
for $k=1,\ldots,N$ and $j=1,\ldots,M.$
\end{definition}
All the polytopes defined above are decorated by  \hskip2.5mm$\widetilde{}\hskip1mm$ to distinguish them from the polytopes $\Delta^t$ appearing in the shrinking procedure. The upper index indicates the dimension of the polytope. Note that if $\Delta$ has a trivial labeling, then every polytope in \ref{def.pol} has a trivial labeling.

The above sets are convex as being intersections of half spaces. However, it is not immediately clear if they are compact.

\begin{lem} \label{lemma djcompact}
 For any $j=1,\ldots, M+1$ and any $a >0$ the polytope
$$ P=P_j(a):=\{x \in \bb{R}^{k_1+\ldots+k_j}\,|\, \langle x, \pi_j^{\perp}(v_i)\rangle <a,\, i \in \mathcal{D}_1 \cup \ldots \cup \mathcal{D}_j \}$$ is compact.
\end{lem}
If $j=M+1$ the map $\pi_{M+1}^{\perp}$ is simply the identity.
\begin{proof}
Suppose the claim is false, that is, suppose there exists a non-zero vector $v \in \bb{R}^{k_1+\ldots+k_j}$ such that $nv \in P$ for any $n \in \bb{N}$. This means that for all $i \in \mathcal{D}_1\cup \ldots \cup \mathcal{D}_j$ and all $n \in \bb{N}$, $\langle n\,v,\,\pi_j^{\perp}(v_i)\rangle\,\leq a$, so $\langle v,\,\pi_j^{\perp}(v_i)\rangle\,\leq 0.$
 Let $$v'=(v,0) \in \bb{R}^{k_1+\ldots+k_{j}} \times \bb{R}^{k_{j+1}+\ldots +k_{M+1}}.$$
 It follows that $ v' \notin \Delta^{t_j}\subset\{0\}\times \bb{R}^{k_{j+1}+\ldots +k_{M+1}}$ while
 $\langle v', v_i\rangle \leq 0$ for all $i \in \mathcal{D}_1 \cup \ldots \cup \mathcal{D}_j$, i.e. $v' \in \bigcap_{l=1}^j\bigcap_{i \in \mathcal{D}_l} h_i^{t_l}.$
 This would lead to a contradiction if $\{1,\ldots,d\}=\mathcal{D}_1 \cup \ldots \cup \mathcal{D}_j$. Assume now that $\{1,\ldots,d\}\setminus\mathcal{D}_1 \cup \ldots \cup \mathcal{D}_j$ is non-empty. For any $\varepsilon >0$ we have $\varepsilon v' \notin \Delta^{t_j}$ whereas $ \varepsilon v' \in \bigcap_{l=1}^j\bigcap_{i \in \mathcal{D}_l} h_i^{t_l}$, so there must exist $i_{\varepsilon}\in \{1,\ldots,d\} \setminus (\mathcal{D}_1 \cup \ldots \cup \mathcal{D}_j)$ such that $ \varepsilon v' \notin h_{i_{\varepsilon}}^{t_{j}}$. In particular, there exists some $i_0 \in \{1,\ldots,d\} \setminus (\mathcal{D}_1 \cup \ldots \cup \mathcal{D}_j)$
 such that for every sufficiently small $\varepsilon>0$ we have $ \varepsilon v' \notin h_{i_0}^{t_{j}}.$ Since $0 \in h_{i_0}^{t_{j}}$, that would imply $0 \in H_{i_0}^{t_{j}}$ which means that $v_{i_0} \in \bb{R}^{k_1+\ldots+k_j}\times\{0\}$ and $i_0 \in \mathcal{D}_j$. Contradiction.
\end{proof}

\begin{lem}
 The polytopes $\widetilde{\Delta}^{n}_{k}$ and $\widetilde{\Delta}^{k_1+\ldots+k_j}$ given by (\ref{def.pol}) are compact.
\end{lem}
\begin{proof}
 The polytopes $\widetilde{\Delta}^{k_1+\ldots+k_j}$ and $\widetilde{\Delta}^{n}_{0}$ are compact due to Lemma \ref{lemma djcompact}. The polytope $\widetilde{\Delta}^{n}_{k}$ is contained in $( \lambda_{j_k}\,/ \,t_{M+1})\,\widetilde{\Delta}^{n}_{0}$, i.e. in the polytope $\widetilde{\Delta}^{n}_{0}$ that is rescaled by $( \lambda_{j_k}\,/ \,t_{M+1}).$ Therefore it is also compact.
\end{proof}

Now that we know that the sets defined in \ref{definition polytopes} are in fact compact polytopes we are ready to prove the Theorem \ref{prva}.
\begin{proof} {\it (of Theorem \ref{prva} )} We need to show that
$$\Delta=\bigcap_{k=0}^N \widetilde{\Delta}^n_k,\hskip1mm \textrm{ if }\hskip1mm M=0$$ and
$$\Delta=\bigcap_{k=0}^N \widetilde{\Delta}^n_k\hskip1mm\cap\bigcap_{j=1}^M(\widetilde{\Delta}^{k_1+\cdots+ k_j}\times\mathbb{R}^{n-(k_1+\cdots+ k_j)}),\hskip1mm \textrm{ if }\hskip1mm M\geq1.$$
We start with proving inclusion "$\subset$'' for both cases.
For every $k=1,\ldots, N$ and any $i \in \mathcal{D}_j$ we have $\lambda_i=t_{j}< \lambda_{j_k}$. Hence,
\begin{align*} \Delta & \subset\bigcap_{i\in\mathcal{D}_1\cup\cdots\cup\mathrm{D}_{M+1}}\{x\in\mathbb{R}^n |\langle
 x,v_i\rangle\leq\lambda_i\}\cap\bigcap_{i\in\mathrm{I}_k}\{x\in\mathbb{R}^n |\langle x,v_i\rangle\leq\lambda_{j_k}\}\\
&\subset\bigcap_{i\in \mathcal{D}_1\cup\cdots\cup\mathrm{D}_{M+1}\cup\mathrm{I}_k}\{x\in\mathbb{R}^n |\langle
 x,v_i\rangle\leq\lambda_{j_k}\}=\widetilde{\Delta}^{n}_{k}.\end{align*}
Furthermore, since $t_{1}<\ldots<t_{M+1}$ it holds:
\begin{align*} \Delta &\subset \bigcap_{i\in\mathcal{D}_1\cup\cdots\cup\mathrm{D}_{M+1}} \{x\in\mathbb{R}^n |\langle x,v_i\rangle\leq\lambda_i=t_{i}\}\\
& \subset
\bigcap_{i\in\mathcal{D}_1\cup\cdots\cup\mathrm{D}_{M+1}} \{x\in\mathbb{R}^n |\langle x,v_i\rangle\leq t_{M+1}\} =\widetilde{\Delta}^n_0.\end{align*}
Recall that, if $M\geq1$, for any $j\in\{1,\ldots,M\}$ and any $i \in \mathcal{D}_1\cup\cdots\cup\mathcal{D}_j$ we have $\pi_j(v_i)=0$ and therefore
 $$\langle x,v_i\rangle=\langle\pi^{\bot}_j(x),\pi^{\bot}_j(v_i)\rangle+\langle\pi_j(x),\pi_j(v_i)\rangle= \langle\pi^{\bot}_j(x),\pi^{\perp}_j(v_i)\rangle.$$
 Thus for any $j\in\{1,\ldots,M\}$ we have
\begin{align*} \Delta &\subset\bigcap_{i\in\mathcal{D}_1\cup\cdots\cup\mathcal{D}_j}\{x\in\mathbb{R}^n |\langle x,v_i\rangle\leq t_{i}\}
\subset\bigcap_{i\in\mathcal{D}_1\cup\cdots\cup\mathcal{D}_j}\{x\in\mathbb{R}^n |\langle x,v_i\rangle\leq t_{j}\}\\
&=\bigcap_{i\in\mathcal{D}_1\cup\cdots\cup\mathcal{D}_j}\{x\in\mathbb{R}^n |\langle \pi^{\bot}_j(x),\pi^{\bot}_j(v_i)\rangle\leq t_{j}\}\\
&=\bigcap_{i\in\mathcal{D}_1\cup\cdots\cup\mathcal{D}_j}\{y\in\mathbb{R}^{(k_1+\cdots+k_j)} |\langle y,\pi^{\bot}_j(v_i)\rangle\leq t_{j}\}\times\mathbb{R}^{n-(k_1+\cdots+k_j)}\\
&=\widetilde{\Delta}^{k_1+\ldots+k_j}\times\mathbb{R}^{n-(k_1+\cdots+k_j)}.\end{align*}
This proves one inclusion.
 To prove the other one, take an arbitrary $x$ in the intersection on the right. Let $i\in\{1,\ldots,d\}$ be an arbitrary index.
 Since $\{1,\ldots,d\}=\mathrm{I}_1\cup\cdots\cup\mathrm{I}_N\cup\mathcal{D}_1\cup\cdots\cup\mathcal{D}_{M+1},$ it follows either $i\in\mathrm{I}_k,$ for some $k\in\{1,\ldots,N\}$ or $i\in\mathcal{D}_j,$ for some $j\in\{1,\ldots,M+1\}.$ If $i\in\mathrm{I}_k,$ since $x\in\widetilde{\Delta}^n_k\subset\bigcap_{i\in\mathrm{I}_k}\{x\in\mathbb{R}^n |\langle x,v_i\rangle\leq\lambda_i=\lambda_{j_k}\},$
 we have $\langle x,v_i\rangle\leq\lambda_i.$ Similarly if $i\in\mathcal{D}_{M+1}$. Since
 $x\in\widetilde{\Delta}^n_0\subset\bigcap_{i\in\mathcal{D}_{M+1}}\{x\in\mathbb{R}^n |\langle x,v_i\rangle\leq\lambda_i=t_{M+1}\},$ we have $\langle x,v_i\rangle\leq\lambda_i.$
If $M\geq1$ and $i\in\mathcal{D}_j,\hskip1mm j\in\{1,\ldots,M\},$ since $x=(\pi^{\bot}_j(x),\pi_j(x))\in\widetilde{\Delta}^{k_1+\cdots+k_j}\times\mathbb{R}^{n-(k_1+\cdots+k_j)},$ it follows
 $\langle\pi^{\bot}_j(x),\pi^{\bot}_j(v_i)\rangle\leq\lambda_i=t_{j}.$ But $\pi_j(v_i)=0$ hence
 $\langle x,v_i\rangle_{\mathbb{R}^n}=\langle\pi^{\bot}_j(x),\pi^{\bot}_j(v_i)\rangle\leq\lambda_i=t_{j}.$
 Since index $i\in\{1,\ldots,d\}$ was arbitrary, we proved $x\in\Delta.$
 \end{proof}

\begin{example}\label{inter.newfirstexample} We apply the above construction to the polytope $\Delta$ given in Example \ref{newfirstexample}. We have $M=0$ and $N=1.$ Therefore
$\Delta=\bigcap_{k=0}^1 \widetilde{\Delta}^2_k,$ where $\mathcal{D}_{M+1}=(\mathcal{D}_1)=\{1,2,3,4,5\},\hskip1mm \mathrm{I}_1=\{6\}$ and
\begin{align*} \widetilde{\Delta}^2_0 =\{x\in\mathbb{R}^2|&\langle x,(-1,0)\rangle\leq3,\langle x,\pm(0,1)\rangle\leq3,\langle x,(1,\pm1)\rangle\leq3 \},\end{align*}
\begin{align*} \widetilde{\Delta}^2_1 =\{x\in\mathbb{R}^2|&\langle x,(-1,0)\rangle\leq4,\langle x,\pm(0,1)\rangle\leq4,\langle x,(1,\pm1)\rangle\leq4, \\
&
\langle x,(2,1)\rangle\leq4 \}.\end{align*}
Note that the polytope $\widetilde{\Delta}^2_0$ is simple but it is not smooth, whereas the polytope $\widetilde{\Delta}^2_1$ is simple and smooth (see Figure \ref{figure inter.newfirstexample}).

 \begin{figure}\centering 
\includegraphics[width=4cm]{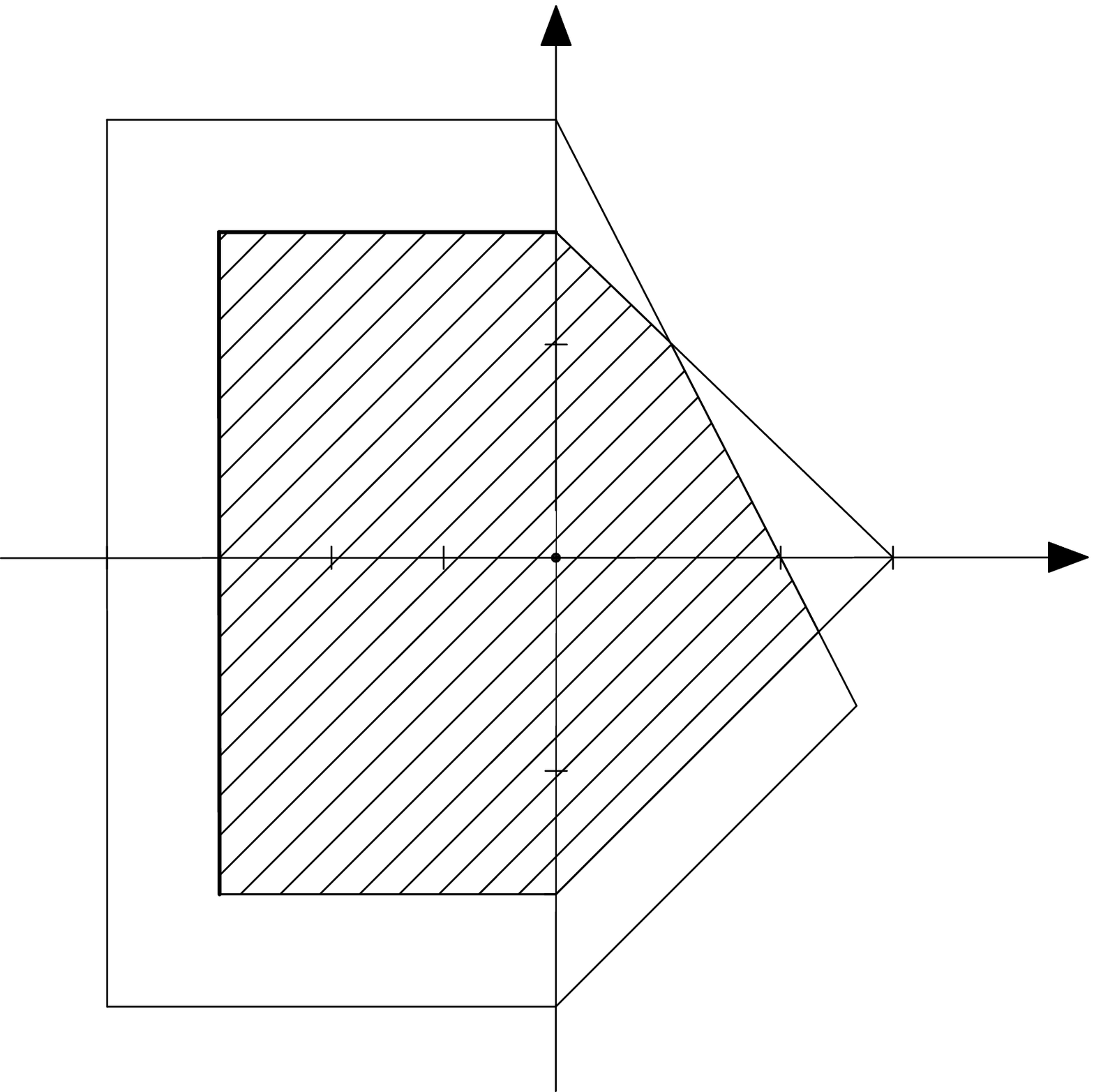}
\caption{Example \ref{inter.newfirstexample}}
\label{figure inter.newfirstexample}
\end{figure}

\end{example}

\begin{example}\label{inter.verynewfirst} The polytope $\Delta$ in Example \ref{verynewfirstexample} has $M=1$ and $N=0.$ Therefore
$\Delta=\widetilde{\Delta}^3_0\hskip1mm\cap(\widetilde{\Delta}^{1}\times\mathbb{R}^2),$ where
 \begin{align*} \widetilde{\Delta}^3_0 =\{x\in\mathbb{R}^3|&\langle x,(\pm1,0,1)\rangle\leq2,\langle x,(0,\pm1,1)\rangle\leq2,\langle x,(0,,0,\pm1)\rangle\leq2 \},\end{align*}
 \begin{align*} \widetilde{\Delta}^1 =\{z\in\mathbb{R}|&\langle z,\pm1\rangle\leq1 \}.\end{align*}
The polytope $\Delta^3_0$ is neither simple nor smooth, whereas the polytope $\Delta^1$ is smooth and simple (see Figure \ref{figure inter.verynewfirst}).

 \begin{figure}\centering 
\includegraphics[width=6.5cm]{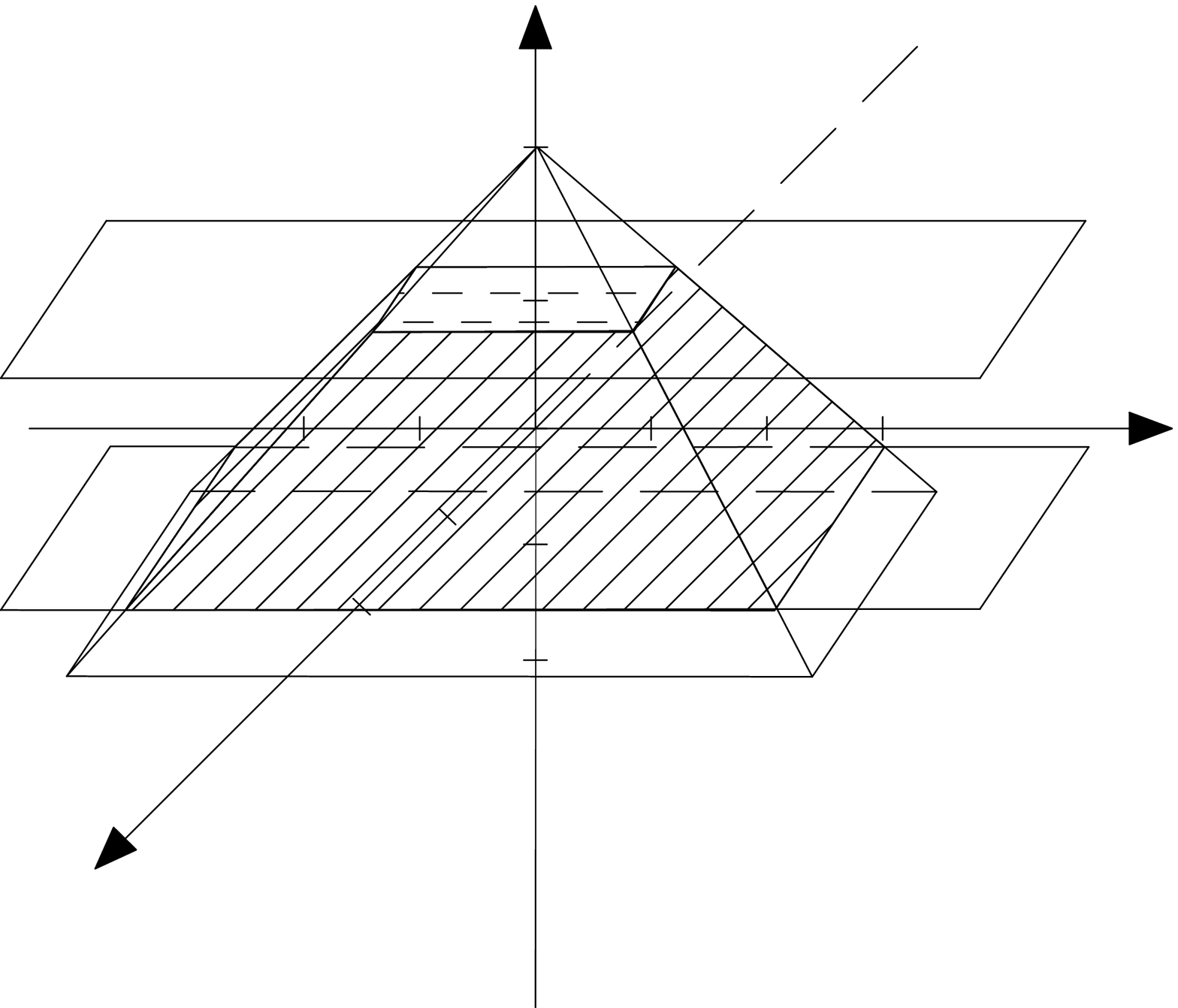}
\caption{Example \ref{inter.verynewfirst}}
\label{figure inter.verynewfirst}
\end{figure}
\end{example}

\begin{example}\label{interthirdexample} The polytope $\Delta$ given in Example \ref{thirdexample} has $M=1$ and $N=1.$ Therefore
$\Delta=\bigcap_{k=0}^1 \widetilde{\Delta}^2_k\hskip1mm\cap(\widetilde{\Delta}^{1}\times\mathbb{R}),$ where
$\mathcal{D}_1=\{1,2\},\hskip1mm \mathcal{D}_2=\{3,5\}$ and
$\mathrm{I}_1=\{4\}.$ Hence
\begin{align*} \widetilde{\Delta}^1 =\{y\in\mathbb{R}|&\langle y,\pm1\rangle\leq2 \},\end{align*}
 \begin{align*} \widetilde{\Delta}^2_0 =\{x\in\mathbb{R}^2|&\langle x,\pm(0,1)\rangle\leq6,\langle x,(1,1)\rangle\leq6,\langle x,(-1,0)\rangle\leq6 \},\end{align*}
 \begin{align*} \widetilde{\Delta}^2_1 =\{x\in\mathbb{R}^2|&\langle x,\pm(0,1)\rangle\leq8,\langle x,(1,1)\rangle\leq8,\langle x,(-1,0)\rangle\leq8,
\langle x,(2,-1)\rangle\leq8 \}\end{align*}
See Figure \ref{figure interthirdexample}. All the polytopes that define the intersection are simple but not necessarily smooth ($\widetilde{\Delta}^2_1$ is not smooth).

 \begin{figure}\centering 
\includegraphics[width=5cm]{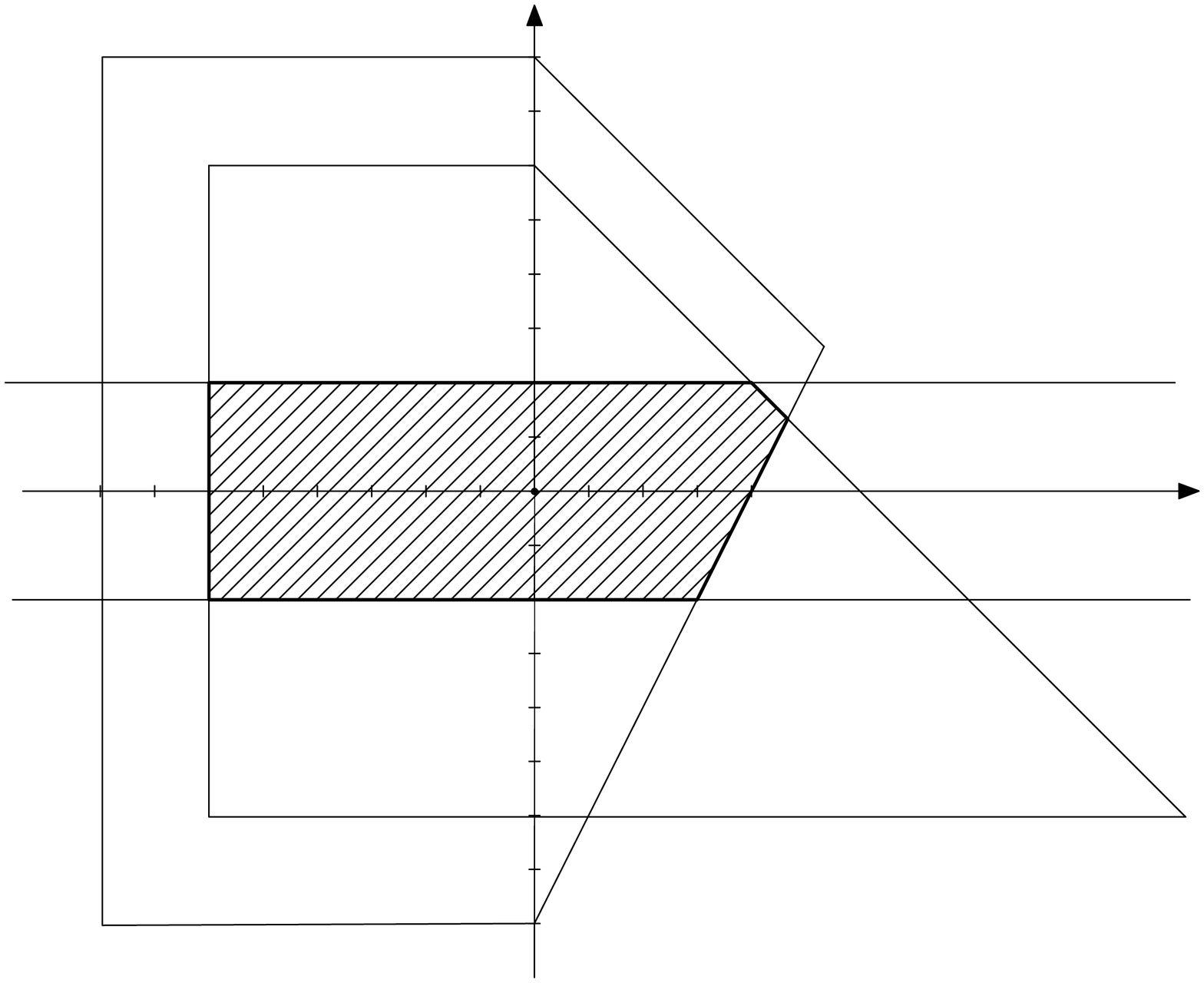}
\caption{Example \ref{interthirdexample}}
\label{figure interthirdexample}
\end{figure}
\end{example}

%%%%%%%%%%%%%%%%%%%%%%%%%%%%%%%%%%%%%%%%%%%%%%%%%%%%%%%%%%%%%%%%%%%
%SECTION CENTERED SYMPLECTIC REDUCTION--------------------
\section{Centered symplectic reduction}\label{section centered reduction}
In Sections \ref{section introduction} and \ref{section shrinking} we explained how a compact symplectic toric orbifold $\M^{2n}$, with $T=(S^1)^n$ torus action, is determined by a simple, labeled polytope, centered at the origin
\begin{equation}\label{polytopedelta}
  \mu(\M)=\Delta=\bigcap_{i=1}^d\{x\in\mathbb{R}^n |\langle x,v_i\rangle\leq\lambda_i\},
\end{equation}
where $v_i=a_i \cdot w_i$, $w_i$ is the primitive outward normal to the $i$-th facet and $a_i \in \bb{Z}_{+}$ is the label on that facet. 
Lerman and Tolman in \cite{LermanTolman:1997} showed that the orbifold $\M$ can be obtained as a symplectic reduction of $(\mathbb{C}^{d},\frac{i}{2}\sum dz_k\wedge d\bar{z}_k)$ with respect to the (not necessarily connected) subgroup $K\subset T^d,$ the kernel of the map from $T^d=\mathbb{R}^{d}/(2\pi \mathbb{Z})^d$ to $T^n=\mathfrak{t}^n/\mathfrak{t}^n_{\mathbb{Z}}= \mathbb{R}^{n}/(2\pi \mathbb{Z})^n$ induced by the linear map given by
\begin{equation}\label{projection}
  \pi(e_i)=v_{i},\hskip1mm i\in\{1,\ldots,d\}
\end{equation}
(so $Lie\,( K)=\ker(\pi)\subset Lie(T^d)\cong \bb{R}^d)$. Precisely, the standard action of the torus $T^d$ on $(\mathbb{C}^{d},\frac{i}{2}\sum dz_k\wedge d\bar{z}_k)$
 $$(t_1,\ldots,t_d)\ast (z_1,\ldots ,z_d)\longmapsto(t_1 z_1,\ldots, t_d z_d)$$
 is a Hamiltonian action whose moment map is of the form 
 $$(z_1,\ldots, z_d) \mapsto  -\frac{1}{2}(|z_1|^2,\ldots,|z_d|^2)+const$$ for some constant (recall that we use the convention $S^1=\mathbb{R}/2\pi\mathbb{Z}).$
 Following the convention in \cite{LermanTolman:1997} we choose 
 $\Phi \colon \mathbb{C}^{d}\rightarrow \mathbb{R}^d$ given by 
 \begin{equation}\label{momentmap}
 \Phi(z_1,\ldots, z_d) = -\frac{1}{2}(|z_1|^2,\ldots,|z_d|^2)+(\lambda_1,\ldots,\lambda_d)
 \end{equation}
 as a moment map. The induced action of the subgroup $K\subset T^d$ on $\mathbb{C}^d$ is Hamiltonian, with a moment map $\iota^{\ast}\circ\Phi$, where $\iota \colon Lie(K)\rightarrow Lie(T^{d})$ is the map induced by inclusion $K \hookrightarrow T^{d}$. 
 The orbifold $\M$ is the quotient of the level $(\iota^{\ast}\circ\Phi)^{-1}(0)$ by the group $K$ and the toric action on it is the action of the residual, $n$-dimensional torus $T^d /K$. 
 Let $p \colon  (\iota^{\ast}\circ\Phi)^{-1}(0) \rightarrow \M$ denote the quotient map.
 As the kernel of $\iota^* \colon Lie(T^d)^* \rightarrow Lie(K)^*$ is naturally isomorphic to $Lie(T^d/K)^*$ we can view $\Phi_{|(\iota^{\ast}\circ\Phi)^{-1}(0)}$ as a map to $Lie(T^d/K)^*$ and uniquely define a moment map $\Phi_T \colon \M \rightarrow Lie(T^d/K)^*$ for the residual $T^d/K$ action on $\M$ by $\Phi_T \circ p = \Phi$ on $(\iota^{\ast}\circ\Phi)^{-1}(0)$.
 If we choose an appropriate identification of $T^d/K$ with $T$, then the polytope $\Phi_T(\M) \subset Lie(T^d/K)^* \cong \mathfrak{t}^* \cong \bb{R}^n$ is exactly $\Delta$ we started with, see \cite{LermanTolman:1997}).
 (Otherwise we would obtain a polytope that differs from $\Delta$ by an $\pm SL(n,\bb{Z})$ transformation.) 
 One interesting property of polytopes that we will use later is contained in the following lemma.
\begin{lem}\label{sum.normals} Let $P=\bigcap_{i=1}^d\{x\in\mathbb{R}^n |\langle x,v_i\rangle\leq\lambda_i\}$ be a compact convex rational (not necessarily simple) labeled polytope.
Then, there are coprime $m_i\in\mathbb{Z}_{+}$ such that $\sum_{i=1}^d m_i v_i=0.$
\end{lem}
\begin{proof} Since $P$ is compact, its associated fan is the whole $\mathbb{R}^n$. Take the integral vector $v:=v_1+\cdots+v_d.$ There exists a vertex $V$ of $P$ such that the vector $v$ belongs to the cone $C_V$ of the fan of $P$. Let $\mathrm{I}_V$ be the set of indices of facets meeting at $V$. As $v$ is integral and $\Delta$ is rational, it follows that $v_1+\cdots+v_d=\sum_{i\in\mathrm{I}_V}\frac{p_i}{q_i}(-v_i),$ for some $\frac{p_i}{q_i}\in\mathbb{Q}_{+}\cup \{0\}.$
Multiplying the above equation by the least common multiple of denominators $q_i, i \in I_V$, we obtain the desired equation.
\end{proof}
\begin{remark} Solution to Minkowski problem for polytopes provides an example of $m_i \in \bb{R}_+$, $i=1,\ldots,d$ satisfying condition $\sum_{i=1}^d m_i v_i=0$ of Lemma \ref{sum.normals}. For any compact convex rational (not labeled) polytope with unit normals $v_1,\ldots,v_d$ one has $0=\sum_{i=1}^d m_i\,v_i$ where $m_i$ is the Euclidean area of the $i-$th facet. This implies that if $v_1,\ldots,v_d$ are primitive normals then one can take $m_i$ to be the symplectic area of the $i$-th facet. For more details see for example a recent work of \cite{Klain:2004} reproving solution to Minkowski problem for polytopes, or references therein.
\end{remark}
 In  \cite[Section 4.3]{AbreuMacarini:2013} Abreu and Macarini show that every symplectic toric manifold corresponding to a monotone Delzant polytope is a centered reduction of a weighted projective space $\mathbb{CP}(1,m_2,\ldots,m_{d})$. We generalize this result and prove the following theorem.
 \begin{prop}\label{mon.symp.red} Every compact symplectic toric orbifold is a symplectic reduction of a weighted projective space. Moreover, if the orbifold corresponds to a monotone labeled polytope then the reduction is centered.
\end{prop}
\begin{proof} Let a compact symplectic toric orbifold $\M$ correspond to a labeled polytope $\Delta$ given by (\ref{polytopedelta}). 
As before, we denote by $K$ the subgroup of $T^d$ for which $Lie(K)=\ker \pi$,  where $\pi$ is given by (\ref{projection}).
Let $K_1$ be a circle in $T^d$ given by the image of the inclusion $t\hookrightarrow(t^{m_1},\ldots,t^{m_d}),$
 where $m_1,\ldots,m_d \in \bb{Z}_+$ are from Lemma \ref{sum.normals}. 
 Then $K_1$ is also a subtorus of $K$.
 Let $K_2\subset T^d$ be a subgroup such that $K=K_1\times K_2.$ Reduce $(\mathbb{C}^{d},\frac{i}{2}\sum dz_k\wedge d\bar{z}_k)$ in two stages. First reduce with respect to the subtorus $K_1\subset K$ and obtain a weighted projective space $\mathbb{CP}(m_1,\ldots,m_{d}).$ Then reduce $\mathbb{CP}(m_1,\ldots,m_{d})$ with respect to $K_2$. Reducing $(\mathbb{C}^{d},\frac{i}{2}\sum dz_k\wedge d\bar{z}_k)$ with respect to $K$ is equivalent to reducing $\mathbb{CP}(m_1,\ldots,m_{d})$ with respect to $K_2$. Therefore, the orbifold $\M$ is a reduction of $\mathbb{CP}(m_1,\ldots,m_{d}).$\\
Assume now that $\Delta$ is monotone, i.e $\lambda_i=\lambda$. As before, (equation \eqref{momentmap}), we take 
$$\Phi(z_1,\ldots, z_d) = -\frac{1}{2}(|z_1|^2,\ldots,|z_d|^2)+(\lambda,\ldots,\lambda)$$
as a moment map for the standard $T^d$-action on $\mathbb{C}^d.$
  Let $\iota_j:Lie(K_j)\rightarrow \mathbb{R}^d$ 
be the map induced by inclusion $K_j\hookrightarrow T^d, \hskip1mm j\in\{1,2\}$.
Then $\iota_1^{\ast}\circ\Phi$ is a moment map for the $K_1$ action on $\bb{C}^d$. Let $$p_1:(\iota_1^{\ast}\circ\Phi)^{-1}(0)\rightarrow (\iota_1^{\ast}\circ\Phi)^{-1}(0)/K_1=\mathbb{CP}(m_1,\ldots,m_{d})$$ be the quotient map. For a moment map for $K_2$ action on $\mathbb{CP}(m_1,\ldots,m_{d})$ take the unique map $\mu_2 \colon \mathbb{CP}(m_1,\ldots,m_{d}) \rightarrow Lie (K_2)^*$  that makes the following diagram commutative
\begin{displaymath}
    \xymatrix{
        {\mathbb{C}^d\supseteq(\iota_1^{\ast}\circ\Phi)^{-1}} \ar[r]^{\hskip8mm\Phi} \ar[d]_{p_1} & \mathbb{R}^{d} \ar[d]^{\iota_2^{\ast}} \\
        {\mathbb{CP}(m_1,\ldots,m_{d})} \ar[r]_{\hskip6mm\mu_2}       & Lie (K_2)^* }
\end{displaymath}
i.e. $\mu_2\circ p_1=\iota_2^{\ast}\circ\Phi$ on $(\iota_1^{\ast}\circ\Phi)^{-1}(0).$ The orbifold $\M$ is the quotient $\mu_2^{-1}(0)/K_2.$  
The central torus fiber of $\mathbb{CP}(m_1,\ldots,m_{d})$ is given by
$$T_{0}=\{[z_1,\ldots,z_{d}]\in\mathbb{CP}(m_{1},\ldots,m_{d})|\hskip1mm |z_1|^2=\cdots=|z_{d}|^2=2\lambda\}.$$
 In order to prove that the reduction $\M=\mathbb{CP}(m_1,\ldots,m_{d})// \,K_2$ is a centered reduction, we have to show that the central torus fiber $T_0$
 is contained in the level $\mu_2^{-1}(0).$ Let $([z_1,\ldots,z_{d}])\in T_0$ be an arbitrary point. Then we have that
\begin{align*}
\mu_2([z_1,\ldots,z_{d}])&=\mu_2(p_1(z_1,\ldots,z_{d}))=\iota_2^{\ast}(\Phi(z_1,\ldots,z_{d}))\\
&=\iota_2^{\ast}(-\frac{1}{2}(2\lambda,\ldots,2\lambda)+(\lambda,\ldots,\lambda))=\iota_2^{\ast}(0)=0.
\end{align*}
The above computation shows how the monotonicity of $\Delta$ implies that the reduction is centered.
 \end{proof}

Now we are ready to prove the second main theorem in this paper.
\subsection{Proof of Theorem \ref{druga}}

Let $\M$ be a compact symplectic toric orbifold. In this subsection we prove Theorem \ref{druga} saying that $\M$ is a centered reduction of a Cartesian product of weighted projective spaces.

\begin{proof}
 Let $\Delta$ be a labeled polytope corresponding to $\M$ via \cite{LermanTolman:1997}. We can assume $\Delta$ is centered (as translation does not change the orbifold). Theorem \ref{prva} implies that the polytope $\Delta$ can be presented as (\ref{broj2}) or (\ref{broj22}). Without a loss of generality we assume that $\Delta=\widetilde{\Delta}_1\cap\widetilde{\Delta}_2$ or $\Delta=\widetilde{\Delta}_1\cap(\widetilde{\Delta}^{k}\times\mathbb{R}^{n-k}),$
  where $\widetilde{\Delta}_{j}=\bigcap_{i=1}^{d_j}\{x\in\mathbb{R}^{n}|\hskip1mm\langle x,v_{j,i}\rangle\leq\lambda_{j}\}$, $j\in\{1,2\}$ and
$\widetilde{\Delta}^{k}=\bigcap_{i=1}^{d_2}\{y\in\mathbb{R}^{k}|\hskip1mm\langle y,\pi_j^\bot(v_{2,i})\rangle\leq\lambda_2\}$
 are monotone compact labeled polytopes and in the second case we have $v_{2,i}=(\pi_j^\bot(v_{2,i}),0)\in \bb{R}^n$ (see Section \ref{section main proof} for details). Therefore
 $$\{y\in\mathbb{R}^{k}|\hskip1mm\langle y,\pi_j^\bot(v_{2,i})\rangle\leq\lambda_2\}\times\mathbb{R}^{n-k}=\{x\in\mathbb{R}^{n}|\hskip1mm\langle x,v_{2,i}\rangle\leq\lambda_{2}\}.$$ Hence, in both cases polytope $\Delta$ is written as an intersection of $d=d_1+d_2$ half spaces (though $\Delta$ may have less than $d$ facets)
 $$\Delta=\bigcap_{j=1}^2\bigcap_{i=1}^{d_j}\{x\in\mathbb{R}^{n}|\hskip1mm\langle x,v_{j,i}\rangle\leq\lambda_{j}\}.$$
 Following the work of Lerman and Tolman we conclude that the orbifold $\M$ is a symplectic reduction of $(\mathbb{C}^{d},\frac{i}{2}\sum dz_k\wedge d\bar{z}_k)$ with respect to the subgroup $K\subset T^d,$ where $Lie\, (K)=\ker(\pi)$ and
 $\pi:\mathbb{R}^{d}\rightarrow \mathbb{R}^n$ is the linear map given by
$$\pi(e_i)=v_{1,i},\hskip1mm i\in\{1,\ldots,d_1\},\,\,\,\,
\pi(e_{d_1+i})=v_{2,i},\hskip1mm i\in\{1,\ldots,d_2\}.$$ 
As before, we take $\Phi \colon \bb{C}^d \rightarrow \bb{R}^d$ given by 
$$\Phi(z_1,\ldots, z_d) = -\frac{1}{2}(|z_1|^2,\ldots,|z_d|^2)+(\underbrace{\lambda_{1},\ldots,\lambda_{1}}_{d_1},\underbrace{\lambda_{2},\ldots,\lambda_{2}}_{d_2})$$ as a moment map (compare with equation \eqref{momentmap}).
 We split the group $K$ into three subgroups $K=K_1\times K_2 \times K_3.$ $K_1$ and $K_2$ are circles that include into $T^d$ in the following ways (respectively)
$$t\hookrightarrow(t^{m_{1,1}},\ldots,t^{m_{1,d_1}},1,\ldots,1)\hskip1mm \textrm{and}\hskip1mm t\hookrightarrow(1,\ldots,1,t^{m_{2,1}},\ldots,t^{m_{2,d_2}}),$$ where
$m_{k,i}\in\mathbb{N},\hskip1mm i\in\{1,\ldots,d_k\}$ are the constants from Lemma \ref{sum.normals} such that
$\sum_{i=1}^{d_k}m_{k,i}v_{k,i}=0,$ for $k\in\{1,2\}.$ $K_3$ is any choice of complementary group.
We perform the reduction by $K$ in three stages. First we reduce by $K_1$ and $K_2$ to obtain $\mathbb{CP}(m_{1,1},\ldots,m_{1,d_1})\times\mathbb{CP}(m_{2,1},\ldots,m_{2,d_2})$ and then reduce it by $K_3$.
Note that we are not performing reductions prescribed by polytopes $\widetilde{\Delta}^{1}$, $\widetilde{\Delta}^{2}$. These polytopes may not be simple. We are just using the information encoded in these polytopes to divide the reduction prescribed by a simple polytope $\Delta$ into stages.
 Here are the details. Let $\iota_j:Lie(K_j)\rightarrow Lie(T^d)$ be the map induced by inclusion $K_j\rightarrow T^d,$ for $j\in\{1,2,3\}.$ Then $\iota_1^{\ast}\circ\Phi$ is a moment map corresponding to Hamiltonian action of the circle $K_1$ on $\mathbb{C}^d.$
 $\mathbb{CP}(m_{1,1},\ldots,m_{1,d_1})\times\mathbb{C}^{d_2}$ is a quotient of the level $(\iota_1^{\ast}\circ\Phi)^{-1}(0)$ by the circle $K_1$. Let $p_1:(\iota_1^{\ast}\circ\Phi)^{-1}(0)\rightarrow \mathbb{CP}(m_{1,1},\ldots,m_{1,d_1})\times\mathbb{C}^{d_2}$ be the quotient map. As a moment map for the $K_2$ action on $\mathbb{CP}(m_{1,1},\ldots,m_{1,d_1})\times\mathbb{C}^{d_2}$ take $\mu_2$ which satisfies $\mu_2\circ p_1=\iota_2^{\ast}\circ\Phi$ on $(\iota^{\ast}_1\circ\Phi)^{-1}(0)$.
This way we obtain $\mathbb{CP}(m_{1,1},\ldots,m_{1,d_1})\times\mathbb{CP}(m_{2,1},\ldots,m_{2,d_2})$ as a quotient of the level $\mu_2^{-1}(0)$ by the circle  $K_2.$ Call the quotient map $p_2:\mu_2^{-1}(0)\rightarrow \mathbb{CP}(m_{1,1},\ldots,m_{1,d_1})\times\mathbb{CP}(m_{2,1},\ldots,m_{2,d_2})$. Finally, as a moment map for $K_3$ action take $\mu_3:\mathbb{CP}(m_{1,1},\ldots,m_{1,d_1})\times\mathbb{CP}(m_{2,1},\ldots,m_{2,d_2})\rightarrow (Lie K_3)^{\ast}$ such that $\mu_3\circ p_2\circ p_1=\iota_3^{\ast}\circ\Phi$ on the intersection $(\iota^{\ast}_1\circ\Phi)^{-1}(0)\cap(\iota^{\ast}_2\circ\Phi)^{-1}(0)$. $\M$ is the quotient of the level $\mu_3^{-1}(0)$ by the group $K_3$. The central torus fiber of $\mathbb{CP}(m_{1,1},\ldots,m_{1,d_1})\times\mathbb{CP}(m_{2,1},\ldots,m_{2,d_2})$ is $T_{0,1}\times T_{0,2},$ where
 $$T_{0,j}=\{[z_1,\ldots,z_{d_j}]\in\mathbb{CP}(m_{j,1},\ldots,m_{j,d_j})|\hskip1mm |z_1|^2=\cdots=|z_{d_j}|^2=2\lambda_{j}\},$$
 $j\in\{1,2\}$. In order to prove that this reduction is centered, we have to show that
$T_{0,1}\times T_{0,2}\subseteq\mu_3^{-1}(0).$
Let $([z^1_1,\ldots,z^1_{d_1}],[z^2_1,\ldots,z^2_{d_2}])\in T_{0,1}\times T_{0,2}$ be an arbitrary point. Then
\begin{align*}
&\mu_3([z^1_1,\ldots,z^1_{d_1}],[z^2_1,\ldots,z^2_{d_2}])=\mu_3(p_2 \circ p_1(z^1_1,\ldots,z^1_{d_1},z^2_1,\ldots,z^2_{d_2}))\\&=\iota_3^{\ast}(\Phi(z^1_1,\ldots,z^1_{d_1},z^2_1,\ldots,z^2_{d_2}))\\
&=\iota_3^{\ast}(-\frac{1}{2}(2\lambda_1,\ldots,2\lambda_1,2\lambda_2,\ldots,2\lambda_2)+(\lambda_1,\ldots,\lambda_1,\lambda_2,\ldots,\lambda_2))\\&=\iota_3^{\ast}(0)=0.
\end{align*}
\end{proof}

\begin{remark}\label{centralfiber} The image of the central fiber of the product of weighted projective spaces under the projection $p_3:\mu_3^{-1}(0)\rightarrow \M$ is exactly the central fiber of $\M$.
Indeed, as explained in the beginning of this section, the moment map $\Phi_T$ for the action of the residual torus $T^d /K $ on $\M$ satisfies $\Phi_T \circ p = \Phi$ on $(\iota^{\ast}\circ\Phi)^{-1}(0)$, where  $p=p_3 \circ p_2 \circ p_1 \colon  (\iota^{\ast}\circ\Phi)^{-1}(0) \rightarrow \M$ is the projection. 
For any $p_3([z^1_1,\ldots,z^1_{d_1}],[z^2_1,\ldots,z^2_{d_2}])\in p_3(T_{0,1}\times T_{0,2}) \subset \M$ it holds that
$$\Phi_T (\,p_3([z^1_1,\ldots,z^1_{d_1}],[z^2_1,\ldots,z^2_{d_2}])\,)=\Phi_T(p(z^1_1,\ldots,z^1_{d_1},z^2_1,\ldots,z^2_{d_2}))$$
$$=\Phi (z^1_1,\ldots,z^1_{d_1},z^2_1,\ldots,z^2_{d_2})=0.$$
\end{remark}

\subsection{Non-displaceability and Proof of Theorem \ref{posled1}}\label{non-displ}
Let $\M$ be a compact toric symplectic orbifold and $\mu$ be a choice of moment map. 
Adding a constant to $\mu$ if necessary, 
we can assume that the moment map image $\Delta=\mu(\M)$ is a polytope centered at the origin. 
A Lagrangian submanifold $L\subset \M$ is {\bf non-displaceable} if for every Hamiltonian diffeomorphism $\varphi \colon \M \rightarrow \M$ we have $\varphi(L)\cap L\neq\emptyset.$
 For each point $p\in\textrm{int}\Delta$ the level $\mu^{-1}(p)$ is a Lagrangian submanifold of $\M,$ because the action of the torus is free on the set that maps under the moment map $\mu$ to $\textrm{int}\Delta$. 
 The goal of this section is to prove Theorem \ref{posled1} saying that the central torus fiber, $\mu ^{-1}(0)$, is a non-displaceable Lagrangian.
\begin{proof}
Abreu and Macarini in \cite[Corollary 3.4 (i)]{AbreuMacarini:2013} showed that if $\M$ is a symplectic reduction of $\widetilde{\M}$ and a Lagrangian torus fiber $T\subset \M$ is a quotient of the corresponding fiber $\widetilde{T}\subset\widetilde{\M}$, where $\widetilde{T}$ is non-displaceable, then $T\subset \M$ is also non-displaceable. Even though Abreu and Macarini work in the setting of smooth toric manifolds, their result generalizes to toric orbifolds.

According to Theorem \ref{druga}, $\M$ is a centered symplectic reduction of a Cartesian product of $M+N$ weighted projective spaces.

The central torus fiber $T_0$ of a weighted projective space $\mathbb{CP}(m_1,m_2\ldots,m_{d})$ is a non-displaceable Lagrangian torus fiber (for $\mathbb{CP}(1,m_2,\ldots,m_{d})$ by the result of Cho-Poddar \cite{ChoPoddar:2012} ; by Gonzales-Woodward \cite{Gonzales.Woodward:2012} in the general case). Moreover from the work of Woodward, \cite{Woodward:2011}, it follows that $T^1_0\times\cdots\times T^{M+N}_0$ is a non-displaceable Lagrangian torus fiber of a Cartesian product of $M+N$ weighted projective spaces. (Though not explicitly, the paper implies that the appropriate version of Floer homology group for $T^1_0 \times T^2_0$ in $\bb{CP}_1 \times \bb{CP}_2$ is a tensor product of Floer homology groups for $T^j_0$ in $\bb{CP}_j$ which are non-zero.)  Hence, the quotient of the level $T^1_0\times\cdots\times T^{M+N}_0$ is a non-displaceable Lagrangian torus fiber in a compact symplectic toric orbifold $\M$. This quotient is exactly the central fiber of $\M$ (see Remark \ref{centralfiber}).
\end{proof}

%%%%%%%%%%%%%%%%%%%%%%%%%%%%%%%%%%%%%%%%%%%%
%------------------------------------------------------section-------Gromov---width----------
\section{Connections with Gromov width}\label{section gromovwidth}
In this Section we explain how Theorem \ref{prva} implies some results about the Gromov width of symplectic toric manifolds.
In 1985 Mikhail Gromov proved his famous Non-squeezing Theorem saying that a ball
$B^{2n}(r)$ of a radius $r$, in a symplectic vector space $\bb{R}^{2n}$ with the usual symplectic structure, cannot be symplectically embedded into $B^2(R)\times \bb{R}^{2n-2}\subset \bb{R}^{2n}$
unless $r\leq R$. This motivated the definition of an invariant called the Gromov width.
Consider the ball of capacity $a$
$$ B^{2n}_a = \Big \{ z \in \bb{C}^n \ \Big | \ \pi \sum_{i=1}^n |z_i|^2 < a \Big \}\subset \bb{R}^{2n} , $$
with the standard symplectic form
$\omega_{std} = \sum dx_j \wedge dy_j$ inherited from $\bb{R}^{2n}$.
The \textbf{Gromov width} of a $2n$-dimensional symplectic manifold $(\M,\omega)$
is the supremum of the set of $a$'s such that $B^{2n}_a$ can be symplectically
embedded in $(\M,\omega)$. It follows from Darboux Theorem that the Gromov width is positive unless $\M$ is a point.

If the manifold $(\M,\omega)$ is equipped with a Hamiltonian action of a torus $T,$ one can use this action to construct explicit embeddings of balls and therefore to calculate the Gromov width. Many such constructions were developed by various authors (see for example: Karshon and Tolman \cite{KarshonTolman:2005} for not necessarily toric actions and \cite{Traynor:1995}, \cite{Schlenk:2005}, \cite{LMS:2013} for toric ones). In what follows we use the result of Latschev, McDuff and Schlenk, \cite[Lemma 4.1]{LMS:2013}, presented here as Proposition \ref{prop diamond}.
As we are to calculate a numerical invariant, a way of identifying the Lie algebra of $S^1$ with the real line $\bb{R}$ is important. Recall from Section \ref{section shrinking} that for us $S^1 = \bb{R} / 2 \pi \bb{Z}$. With this convention the moment map for the standard $S^1$ action on $\bb{C}$ by rotating with speed $1$ is given (up to a translation by a constant) by $z \mapsto -\frac 1 2 |z|^2$.  Define
$$ \Diamond^n(a):= \left\{ (x_1,\ldots,x_n) \in \bb{R}^n ({\bf{x}}) \,|\, \sum_{j=1}^{n} |x_j| < \frac a 2\right\}\subset \bb{R}^n ({\bf x}).$$ 
If $\M^{2n}$ is toric, $\mu$ is the associated moment map and $\Diamond (a) \subset  \textrm{Int}\mu(\M)$ is a subset of the interior of the moment map image, then a subset of $\mu^{-1}(\Diamond^n(a))\cong \Diamond^n(a) \times T^n$ is symplectomorphic to $\Diamond^n(a) \times (0,2\pi)^n \subset \bb{R}^n({\bf x}) \times \bb{R}^n({\bf y})$ with the symplectic structure induced from the standard one on $\bb{R}^n({\bf x}) \times \bb{R}^n({\bf y})$. 
Below we present a result of Latschev, McDuff and Schlenk, \cite[Lemma 4.1]{LMS:2013} which, although stated in dimension $4$, holds also in higher dimensions.
There the authors were using the convention that $S^1= \bb{R} / \bb{Z}$. To translate between the conventions observe that $\Diamond^n(a) \times (0,2\pi)^n$ is symplectomorphic to $\Diamond^n(2 \pi a) \times (0,1)^n\subset \bb{R}^n({\bf x}) \times \bb{R}^n({\bf y})$.

Note that ``twisting" the $T$ action on $\M$ by an orientation preserving automorphism of $T$, which obviously does not affect the Gromov width of $\M$, changes the moment map image by an $SL(n,\bb{Z})$ transformation.

\begin{prop} \cite[Lemma 4.1]{LMS:2013} \label{prop diamond}
 For each $\varepsilon >0$ the ball $B^{2n}_{2 \pi (a- \varepsilon)}$ of capacity $2 \pi (a- \varepsilon)$ symplectically embeds into $\Diamond^n(a) \times (0,2\pi)^n \subset \bb{R}^n({\bf x}) \times \bb{R}^n({\bf y})$. Therefore, if for a toric manifold $(\M^{2n},\omega)$ with moment map $\mu$, we have $\Psi(\Diamond^n(a)) \subset \textrm{Int}\mu(\M)$ for some $\Psi \in SL(n,\bb{Z})$, then the Gromov width of $(\M^{2n}, \omega)$ is at least $2 \pi\,a$.
\end{prop}
The analysis of a moment map image we did in Section \ref{section shrinking} allows us to notice certain diamonds $\Diamond^n(a)$ inside the moment map image.
Let $\M$ be any compact symplectic toric manifold. Following Theorem \ref{prva} we present the corresponding Delzant polytope $\Delta=\{x \in \bb{R}^n\,|\, \langle x, v_j\rangle \leq \lambda_j,\,j=1,\ldots,d\}$ as an intersection of polytopes described in Definition \ref{definition polytopes}. We continue to denote by $t_1$ the time of the first  (possibly unique) ``dimension drop''.  Notice that 
$$\Delta \supset t_1 \cdot \Delta':=\{t_1\cdot x\,|\, x \in \Delta'\}\textrm{ where }\Delta'=\{x \in \bb{R}^n\,|\, \langle x, v_j\rangle \leq 1,\,j=1,\ldots,d\}.$$
Rational polytopes with the facet presentation of $\Delta'$, i.e., monotone, with coefficient $\lambda=1$, are called reflexive (see for example \cite[Definition 2.3.11]{CLS}).

The (dual version of the) Ewald Conjecture says that for any reflexive polytope $\Delta'$, of dimension $n$, the set $\{v \in \bb{Z}^n \cap \Delta '\,|\,-v \in \bb{Z}^n \cap \Delta '\,\}$ contains some integral vectors $v_1, \ldots, v_n$ that form a basis of $\bb{Z}^n$ (see \cite[Section 3.1]{MD} for the dual version, and \cite[Section 4]{O} for the usual version). 
The convex hull of $\{ \pm v_1, \ldots, \pm v_n\}$ is $SL(n,\bb{Z})$ equivalent to a diamond $\Diamond^n(2)$. Therefore the Ewald Conjecture would imply that $\Diamond^n(2t_1) \subset \Delta$, proving that the Gromov width of $\M$ is at least $4 \pi t_1$ (via Proposition \ref{prop diamond}). The Ewald Conjecture has been verified by \O bro, \cite{O}, for polytopes of dimensions $\leq 8$. 
In higher dimensions it remains an open question. The above argument proves the following corollary.

\begin{cor}\label{cor lowerbound}
Let $\M$ be any compact symplectic toric manifold of dimension less or equal to $ 16$.  
 The Gromov width of $\M$ is at least $4 \pi t_1$, where $t_1$ is the time of the first  (possibly unique) ``dimension drop''.
Moreover, if the Ewald Conjecture turns out to be true, one can remove the assumption $\dim \M \leq 16$ from the above statement.
\end{cor}
 
Figure \ref{figure goodlowerbound} presents few examples where this general lower bound for the Gromov width in fact gives the actual Gromov width (use  Proposition \ref{prop upperbound} below to find the actual Gromov width). Figure \ref{figure badlowerbound} shows an example of $\bb{CP}^2$ where the above lower bound is far from the actual Gromov width.

\begin{figure}[h]\centering
\includegraphics[width=7cm]{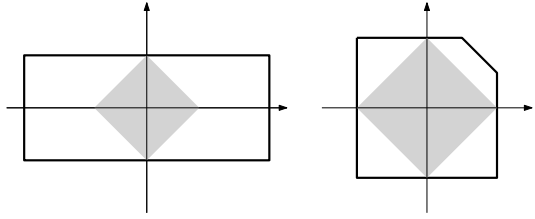}
\caption{Two examples where the lower bound obtained from Corollary \ref{cor lowerbound} gives the actual Gromov width.}
\label{figure goodlowerbound}
\end{figure}

\begin{figure}[h]\centering
\includegraphics[width=3cm]{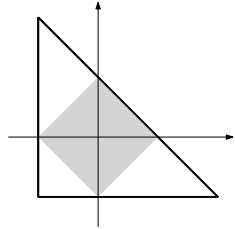}
\caption{An easy example where the lower bound obtained from Corollary \ref{cor lowerbound} is lower than the actual Gromov width.}
\label{figure badlowerbound}
\end{figure}

To find the actual Gromov width one also needs some information about the upper bounds. Here we quote a result of Lu \cite{Lu:2006}. 
\begin{prop}\cite{Lu:2006} \label{prop upperbound}
 Suppose a compact symplectic toric manifold $(\M, \omega)$ is also Fano, i.e. the anticanonical line bundle is ample. Let the Delzant polytope $\Delta=\bigcap_{i=1}^d\{x\in\mathbb{R}^n |\langle x,w_i\rangle\leq\,l_i\},$ be its moment map image, with $w_1, \ldots, w_d$ being the primitive outward normals.
  Then the Gromov width of $(\M, \omega)$ is at most
$$\inf \left \{ \sum_{j \in J} 2 \pi \, a_j l_j \,\,|\,\,a_j \in \bb{Z}_{> 0},\, \sum_{j \in J} a_j w_j=0, \, J \subset \{1,\ldots, d\} \right\}.$$
\end{prop}
Therefore, if $\Delta = I \times \Delta'$, i.e. $w_{j_1}+w_{j_2}=0$ for some $j_1,j_2 \in \{1,\ldots,d\}$, then the Gromov width of the corresponding toric manifold is at most 
$2 \pi (l_{j_1}+l_{j_2})$.
This argument shows that the lower bound coming from \ref{cor lowerbound} is equal to the actual Gromov width for the examples presented in Figure \ref{figure goodlowerbound}, as well as in Examples \ref{verynewfirstexample} and \ref{secondexample}, and proves the following corollary.
\begin{cor}
 Let $(\M^{2n}_{\Delta}, \omega)$ be a compact toric symplectic Fano manifold of dimension $2n$. Suppose that during the shrinking procedure for $\Delta$ dimension drops only once 
 (monotone case) and that $\widetilde{\Delta}^n_0$ from  Definition \ref{definition polytopes}  is a product $\widetilde{\Delta}^n_0=I \times \Delta'$, or,  that the first dimension drop is by $k_1 <n$ and the polytope $\widetilde{\Delta}^{k_1}$ is a product $\widetilde{\Delta}^{k_1}=I \times \Delta'$. Then the Gromov width of $(\M^{2n}_{\Delta}, \omega)$ is equal to $4 \pi \,t_1$, where $t_1$ is the time of the first ``dimension drop" in the shrinking procedure.
\end{cor}

%----------------------------------------

\end{document}